\documentclass[11pt,a4paper,twoside]{article}
\usepackage{amssymb}
\usepackage{amsfonts}
\usepackage{geometry}
\usepackage{color}
\usepackage{amsmath,amsthm,amssymb,amsfonts}
\usepackage{CJK}
\usepackage{latexsym}
\usepackage{graphicx}
\usepackage{pgfpages}
\usepackage{mathrsfs}
\usepackage{ytableau}
\usepackage{enumerate}
\usepackage[affil-it]{authblk}
\allowdisplaybreaks

\newtheorem{definition}{Definition}[section]
\newtheorem{theorem}[definition]{Theorem}
\newtheorem{lemma}[definition]{Lemma}
\newtheorem{proposition}[definition]{Proposition}
\newtheorem{corollary}[definition]{Corollary}
\newtheorem{remark}[definition]{Remark}
\newtheorem{example}[definition]{Example}

\numberwithin{equation}{section}

\begin{document}
	\title{A Sufficient and Necessary Condition of PS-ergodicity of Periodic Measures and Generated Ergodic Upper Expectations}
	\author[1,2]{Chunrong Feng}
	\author [3,1] {Baoyou Qu}
	\author[1,2,4]{Huaizhong Zhao}
	\affil[1]{Department of Mathematical Sciences, Loughborough
		University, LE11 3TU, UK}
	\affil[2]{Department of Mathematical Sciences, Durham
		University, DH1 3LE, UK}
	\affil[3] {Zhongtai Securities Institute for Financial Studies, Shandong University, Jinan 250100, China}
	\affil[4]{Research Centre for Mathematics and Interdisciplinary Sciences, Shandong University, Qingdao 266237, China}
	\affil[ ]{C.Feng@lboro.ac.uk, qu@mail.sdu.edu.cn, H.Zhao@lboro.ac.uk}
	\date{}
\maketitle

\begin{abstract}
	This paper contains two parts. In the first part, we study the ergodicity of periodic measures of random dynamical systems on a separable Banach space. We obtain that the periodic measure of the continuous time skew-product dynamical system generated by a random periodic path is ergodic if and only if the underlying noise metric dynamical system at discrete time of integral multiples of the period is ergodic. For the Markov random dynamical system case, we prove that the periodic measure of a Markov semigroup is PS-ergodic if and only if the trace of the random periodic path at integral multiples of period either entirely lies on a Poincar\'{e} section or completely outside a Poincar\'{e} section almost surely. In the second part of this paper, we construct sublinear expectations from periodic measures and obtain the ergodicity of the sublinear expectations from the ergodicity of periodic measures. We give some examples including the ergodicity of the discrete time Wiener shift of Brownian motions. The latter result would have some independent interests.
	
\end{abstract}

\section{Introduction}
  Ergodic theory is one of the most important observations in mathematics made in the last century with significance in many areas of physics such as statistical physics (c.f. \cite{Birkhoff},\cite{Neumann1},\cite{Neumann2},\cite{Walters}). Concerning the spreading and irreducibility nature of random systems, ergodicity is a fundamentally natural phenomenon to many stochastic systems (c.f. \cite{Da Prato},\cite{Doob},\cite{Durrett},\cite{Hasminskii}). The theory was substantially developed for (weakly) mixing invariant measures/stationary processes in the stationary regime. Many useful results were obtained especially in the Markovian case.
 
It is noted that the classical ergodic theory excludes random periodic cases. However, random periodicity is ubiquitous. Randomness and periodicity are present simultaneously in many real world problems e.g. maximum daily temperature, sunspot activities, economic cycles, business cycles, El Nino phenomena, ice age and interglacial transitions etc. The concept of random periodic paths and periodic measures were recently introduced to describe random periodicity 
(\cite{ZZ},\cite{FZZ},\cite{F.Z3},\cite{F.Z2}). See also \cite{FLZ},\cite{FWZ}. 

It is worth mentioning that the notion of classical periodic paths cannot be adopted to interpret random periodicity as periodicity breaks under random perturbations. Unless in a very restrictive circumstance, a random process cannot follow a periodic path even if one fixes a realization. The idea of random periodic paths is different. Employing the notion of random dynamical systems and its underlying measure preserving metric dynamical systems $(\Omega, \mathcal{F}, P, (\theta_{t})_{t\in I})$, a random periodic path of period $\tau$ is a random path satisfying for $P-a.e. \omega\in \Omega$, 
\begin{eqnarray*}
Y(s+\tau,\omega)=Y(s,\theta_\tau \omega),\ {\rm for\ all}\ s\in I,
\end{eqnarray*}
or equivalently for $P-a.e. \omega\in \Omega$, $Y(s+\tau,\theta_{-\tau}\omega)=Y(s,\omega),\ {\rm for\ all}\ s\in I$. Here $I$ is a two-sided time set of discrete or continuous type,$I={\mathbb Z}$, or $I={\mathbb R}$.
This is very different from the classical periodic function. 
It is only "periodic" if one prepares to kick the noise backward. However, the pull-back process 
$\phi(s,\omega)=Y(s, \theta_{-s}\omega)$ is a periodic function of $s$. Note $(\phi(s,\omega))_{s\in I}$ is not a path of the random dynamical system but the set 
\begin{equation}
	\label{path of pullback process}
	L^{\omega}=\{\phi(s, \omega): 0\leq s\leq \tau\}
\end{equation}
is an invariant set of the random dynamical system $\Phi$ such that $\Phi(t,\omega)L^{\omega}=L^{\theta_t\omega}$. The set $L^{\omega}$ is a closed curve if $I$ is continuous. This suggests that the random periodic path actually moves from one closed curve to another closed curve. Moreover, the law $\rho_s$ of $Y(s,\cdot)$ defined by 
$$\rho_s(\Gamma)=P(\omega: Y(s,\omega)\in \Gamma) \ \rm {for\ any} \ \Gamma\in {\cal B}(\mathbb X)$$
is a measure-valued function satisfying periodic condition $\rho_{s+\tau}=\rho_s$, for any $s\in \mathbb R$.

The ergodic theory of periodic random dynamical system was observed recently in \cite{F.Z2}. The concept of periodic measure was introduced and its ``equivalence" with random periodic processes was established. Moreover, the average of a periodic measure over one period is an invariant measure from which the ergodicity can be studied. It is defined as the ergodicity of a measure preserving canonical dynamical system lifted from the invariant measure on the phase space. From the ergodic theory of periodic random dynamical systems in \cite{F.Z2}, the distinction between the stationary regime and random periodic regime is characterised by the spectral structure of Markov semigroups or their infinitesimal generators. In the stationary regime, Koopman-von Neumann theorem tells us that 0 is a simple and unique eigenvalue of the generator on the imaginary axis (c.f. \cite{Da Prato}), while in the periodic regime, the infinitesimal generator has infinitely many equally spaced eigenvalues (including 0), which are simple, and no other eigenvalues, on the imaginary axis (\cite{F.Z2}).
  
Random periodic paths have been found in many stochastic systems. We will quote two examples in Section 2 to demonstrate the idea. Needless to say that the existence of periodic measures can be studied without referring to random periodic paths. See \cite{FZZ19} for recent progress.
These concepts 
have been used in the study of bifurcations (\cite{wang}), random attractors ({\cite{blw}), stochastic resonance (\cite{CLRS}), modelling El Nino phenomena 
({\cite{chekroun}) and strange attractors ({\cite{hl}). 

  In this paper, we continue the study on the ergodicity of periodic measures and obtain some new results. First we study random periodic paths on a separable Banach space and associated periodic measures $\mu_s=\delta_{Y(s, \theta_{-s}\omega)}\times dP$ on the product space $(\bar{\Omega}, \bar{\mathcal{F}})=(\Omega\times \mathbb{X}, \mathcal{F}\otimes \mathcal{B}(\mathbb{X}))$. Then $\{\mu_s\}_{s\geq 0}$ is a periodic measure with respect to the skew product $(\bar{\Theta}_t)_{t\geq 0}$. We prove in the first part of the paper that $(\Omega, \mathcal{F}, P, (\theta_{n\tau})_{n\geq 0})$ is ergodic if and only if $(\bar{\Omega}, \bar{\mathcal{F}}, \mu_s, (\bar{\Theta}_{n\tau})_{n\geq 0})$ is ergodic, i.e. the dynamical system $(\bar{\Omega}, \bar{\mathcal{F}}, \{\mu_{s}\}_{s\in \mathbb{R}}, (\bar{\Theta}_{t})_{t\geq 0})$ is PS-ergodic. Note here there is no need of any other conditions on its random periodic paths $Y$ apart from the existence.
  
  The metric dynamical system $(\Omega, \mathcal{F}, P, (\theta_{n\tau})_{n\geq 0})$ being ergodic is stronger than the statement that $(\Omega, \mathcal{F}, P, (\theta_{t})_{t\geq 0})$ is ergodic. They are not normally equivalent. We will give an example that $(\Omega, \mathcal{F}, P, (\theta_{t})_{t\geq 0})$ is ergodic, but the discrete metric dynamical system $(\Omega, \mathcal{F}, P, (\theta_{n\tau})_{n\geq 0})$ is not ergodic. However, we will prove in this paper, for the canonical Wiener process and the Brownian shift, both the discrete dynamical systems and continuous time dynamical systems are ergodic. The result of the discrete dynamical systems of Wiener space is new. This means that our results can apply to 
 stochastic differential equations and stochastic partial differential equations driven by Wiener processes. Suggested by fundamental results of 
 \cite{Arnold},\cite{Elworthy},\cite{Flandoli},\cite{Kunita},\cite{Meyer},\cite{Mohammed}, these equations can generate random dynamical systems, of which the noise metric dynamical system over a group of discrete time is ergodic according to our result here.
  
  For the Markov random dynamical systems, the random periodic paths give rise to periodic measures $(\rho_s)_{s\in \mathbb{R}}$ on the state space $\mathbb{X}$. For each $s$, $\rho_s$ is an invariant measure with respect to discrete semigroup $P(n\tau), n\in \mathbb{N}$ (as a convention, we always assume that $0\in \mathbb{N}$). We will give a necessary and sufficient condition for the periodic measure $(\rho_s)_{s\in \mathbb{R}}$ being PS-ergodic (i.e. for each $s\in \mathbb{R}$, $\rho_s$ is ergodic as an invariant measure with respect to $P(n\tau), n\in \mathbb{N}$), which says for any invariant set $\Gamma$ such that $P_{\tau}\I_{\Gamma}=\I_{\Gamma}$ $\rho_s-a.s.$, the section $L_s^{\omega}:=\{Y(s+k\tau, \omega), k\in \mathbb{Z}\}$ satisfies $L_s^{\omega}\subset \Gamma$ or $L_s^{\omega}\cap \Gamma=\emptyset$, $P-a.s.$. However, it is not known whether or not this result is true in the stationary case. 
  
  Sublinear expectation is used to model uncertainty and ambiguity of probabilities such as subjective probabilities due to heterogeneity of expectation formation process (c.f. \cite{Artzner},\cite{Delbaen},\cite{Peng}). An ergodic theory of sublinear expectation was developed recently by \cite{F.Z1}. In the second part of this paper, we construct an ergodic sublinear expectation from an ergodic periodic measure as an upper expectation for the first time in literature. We prove that if a periodic measure is ergodic, then the generated sublinear expectation, which is invariant with respect to the skew product dynamical system or the Markov semigroup, is ergodic. As for the Birkhoff's type of ergodic theorem, i.e. the law of large number, we obtain the convergence in the quasi-sure sense when we apply the ergodic theory of upper expectations, whilst we can only obtain the convergence in the almost-sure sense by the ergodic theory of periodic measures (\cite{F.Z2}). This provides justifications for the construction of upper expectation and the investigation of its ergodicity, which can provide useful new information. The point of view of upper expectations from periodic measures could also be interesting to the study of finance or coherent risk measure. 
 
  \section{Ergodicity of skew product dynamical systems: necessary and sufficient conditions}
 \subsection{Random periodic paths and periodic measures on product spaces }
  
  Consider a random dynamical system $\Phi: \mathbb{R^+}\times \Omega\times \mathbb{X} \rightarrow \mathbb{X}$ over a metric dynamical system $(\Omega, \mathcal{F}, P, (\theta_t)_{t\in \mathbb{R}})$ on a separable Banach space $\mathbb{X}$. It is a measurable mapping and almost surely $\Phi_0=id$ and $\Phi(t+s, \omega)=\Phi(t, \theta_s\omega)\circ \Phi(s, \omega)$ for any $t,s\geq 0$. The map $\theta: \mathbb{R}\times \Omega \rightarrow \Omega$ is measurable with respect to $(\mathcal{B}(\mathbb{R})\otimes\mathcal{F}, \mathcal{F} )$ such that $\theta_t\circ\theta_s=\theta_{t+s}, t, s\in \mathbb{R}$  and preserves the measure $P$, i.e. $\theta_sP=P$. Random dynamical systems can be generated by stochastic differential equations (\cite{Arnold}, \cite{Elworthy}, \cite{Kunita}, \cite{Meyer}), stochastic partial differential equations (\cite{Flandoli}, \cite{Garrido}, \cite{Mohammed}) and Markov chains (\cite{Kifer}).
  
  Let us recall the definition of random periodic paths (\cite{FZZ}, \cite{F.Z3}, \cite{F.Z2}, \cite{ZZ}).
  
  \begin{definition}
  	\label{Def of RPP}
  	A random periodic path of period $\tau$ of the random dynamical system $\Phi : \mathbb{R}^+ \times \Omega \times \mathbb{X} \rightarrow  \mathbb{X}$ is an $(\mathcal{B}(\mathbb{R})\otimes\mathcal{F}, \mathcal{B}(\mathbb{X}))$-measurable map $Y:\mathbb{R} \times \Omega \rightarrow \mathbb{X}$ such that for almost all $\omega \in \Omega$,
  	\begin{equation}
  	\label{RPP}
  	\Phi(t, \theta_{s}\omega)Y(s, \omega)=Y(t+s, \omega),    Y(s+\tau, \omega)=Y(s, \theta_{\tau}\omega), \text{ for all } t\geq 0, s\in \mathbb{R}.
  	\end{equation}
  	It is called a random periodic path with the minimal period $\tau$ if $\tau >0$ is the smallest number such that (\ref{RPP}) holds.
  	It is a stationary path of $\Phi$ if $Y(s, \theta_{-s}\omega)=Y(0, \omega)=:Y_{0}(\omega)$ for all $s \in \mathbb{R}, \omega \in \Omega$ i.e. $Y_{0}: \Omega \rightarrow \mathbb{X}$ is a stationary path if for almost all $\omega \in \Omega$,
  	\begin{equation}
  	\Phi (t, \omega)Y_{0}(\omega)=Y_0(\theta_{t}\omega), \text{ for all } t \in \mathbb{R}^+.
  	\end{equation}
  \end{definition} 
  
 As we mentioned in the introduction, random periodicity is a common phenomenon. We quote the following two examples for convenience. Detailed proof can be found in  \cite {F.Z2} and is omitted here.
\begin{example}\label{exp2.2}
Consider the following stochastic differential equation on ${\mathbb R}^2$
\begin{eqnarray}\label{zhao300x}
\left\{
\begin{array}{cc}
dx_1=&[-x_2+x_1(1-x_1^2-x_2^2)]dt+x_1dW_1(t),\\
dx_2=&[x_1+x_2(1-x_1^2-x_2^2)]dt+x_2dW_2(t).
\end{array}
\right.
\end{eqnarray}
Here $W_1(t)$ and $W_2(t)$ are two independent one-dimensional two-sided Brownian motions on the probability space
$(\Omega, {\mathcal F}, P)$ with   $(W_1(0),W_2(0))^T=(0,0)^T$. Denote $W(t)=(W_1(t),W_2(t))^T$. 
Set ${\mathcal F}_s^t=\sigma(W(u)-W(v): s\leq v\leq u\leq t)$, ${\mathcal F}_{-\infty}^t=V_{s\leq t}{\mathcal F}_s^t$
and $\theta: {\mathbb R}\times \Omega\to \Omega$ the measure preserving metric dynamical system  
given by 
$$(\theta_s\omega)(t)=W(t+s)-W(s), s,t\in {\mathbb R}.
$$
It is well known that the noiseless system 
\begin{eqnarray*}\label{13Jan7}
\left\{
\begin{array}{cc}
{dx_1\over dt}=&-x_2+x_1(1-x_1^2-x_2^2),\\
{dx_2\over dt}=&x_1+x_2(1-x_1^2-x_2^2),
\end{array}
\right.
\end{eqnarray*}
has a periodic solution $(x_1(t),x_2(t))=(\cos t, \sin t)$. It is proved in \cite{F.Z2} that 
Equation (\ref{zhao300x}) has a unique random periodic solution $x^*(t)=(x^*_1(t),x^*_2(t))\ne (0,0)$
with a positive minimum period satisfying for a.s. $\omega\in \Omega$,
\begin{eqnarray}
x^*(t+\pi,\omega)&=&-x^*(t,\theta _{\pi}\omega),\label{zhao310b},\\
x^*(t+2\pi,\omega)&=&x^*(t,\theta _{2\pi}\omega).\label{zhao310c}
\end{eqnarray}
Numerical simulations are also borrowed from  \cite {F.Z2} to provide numerical evidence of the main result (\ref{zhao310b}) and (\ref{zhao310c}) of 
this example (Figure 1). They describe the random periodicity in the sense of backward kicked noise.

\begin{figure}
\begin{minipage}{\textwidth}
\centering 
\includegraphics[scale=0.28]{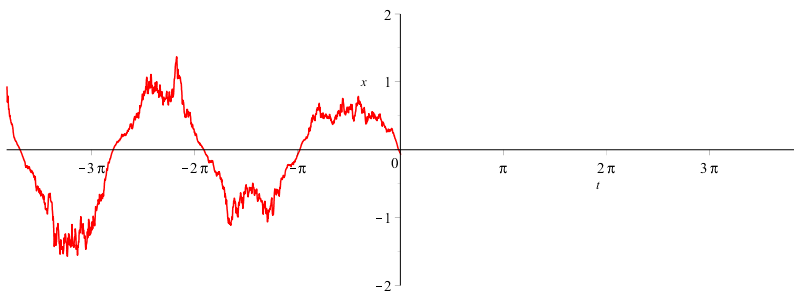}\\
\includegraphics[scale=0.28]{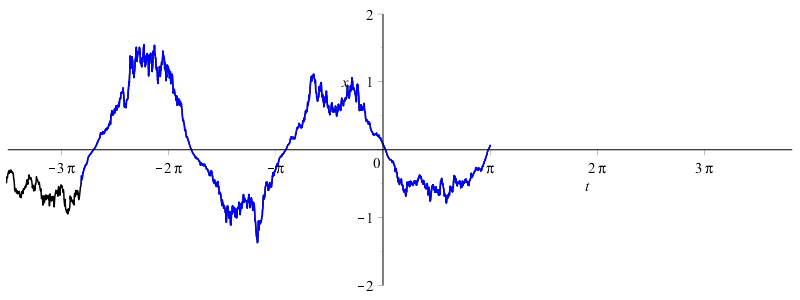}\\
\includegraphics[scale=0.28]{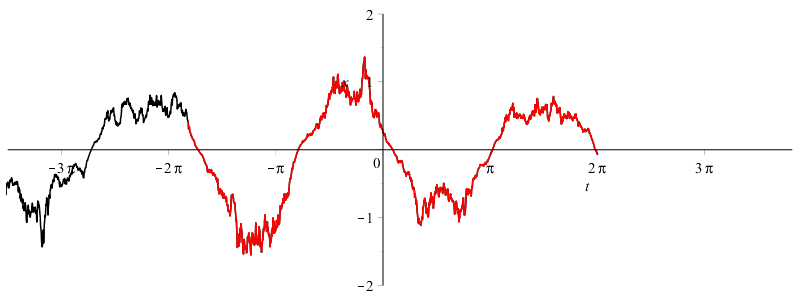}
\end{minipage} 
\caption{ 
From the top to bottom, first coordinate of random periodic paths with one realisation $\omega$, its pullbacks $\theta _{-\pi}\omega$ and $\theta_{-2\pi}\omega$ respectively. Red paths are identical up to a shift and the blue path is the flipped over image of the red paths up to a shift.
}
\label{graph of example 1 samples}
\end{figure}

\end{example}

\begin{example}

Consider the following well known example of a discrete time Markov chain with three states $\{1,2,3\}$ and 
the transition probability matrix
\[ P = \left( \begin{array}{ccc}
0 & {1\over 2} & {1\over 2} \\
1 & 0 & 0 \\
1 & 0 & 0 \end{array} \right).\]
Recall that in the theory of Markov chain the period $d(i)$ of the state $i$ is defined as the greatest common divisor of $\{n: P_{ii}^n>0\}$.
From this definition, it is easy to see that $d(1)=d(2)=d(3)=2$ in this case.
Definition \ref{Def of RPP} looks completely different from the greatest common divisor definition. However, it was shown in \cite{F.Z2} that these two definitions are equivalent.
We can set up a random dynamical system from this Markov chain and construct a random periodic path.
\end{example}

Now we introduce the idea of periodic measures on product space generated by random periodic paths.   Consider a standard product measurable space $(\bar{\Omega}, \bar{\mathcal{F}})=(\Omega \times \mathbb{X}, \mathcal{F}\otimes \mathcal{B}(\mathbb{X}))$ and the skew-product of the metric dynamical system $(\Omega, \mathcal{F}, P, (\theta_{t})_{t\in \mathbb{R}})$ and the cocycle $\Phi(t, \omega)$ on $\mathbb{X}$, 
  $\bar{\Theta}_{t}: \bar{\Omega}\rightarrow \bar{\Omega}$
  \begin{equation}
  \bar{\Theta}_{t}(\bar{\omega})=(\theta_{t}\omega, \Phi(t,\omega)x),  \text{ for all } \bar{\omega}=(\omega, x)\in \bar{\Omega}, t \in \mathbb{R}^+.
  \end{equation}
  Set
  $$\mathcal{P}_{P}(\Omega \times \mathbb{X}):=\{\mu: \text{probability measure on } (\Omega \times \mathbb{X}, \mathcal{F}\otimes \mathcal{B}(\mathbb{X})) \text{ with marginal } P \text{ on } (\Omega, \mathcal{F})\}$$
  and
  $$\mathcal{P}(\mathbb{X})=\{\rho: \text{ probability measure on } (\mathbb{X}, \mathcal{B}(\mathbb{X}))\}.$$
  
  The following definition was given in \cite{F.Z2}.
  
  \begin{definition}
  	A map $\mu: \mathbb{R}\rightarrow \mathcal{P}_{P}(\Omega \times \mathbb{X})$ is called a periodic probability measure of period $\tau$ on $(\Omega \times \mathbb{X}, \mathcal{F}\otimes \mathcal{B}(\mathbb{X}))$ for the random dynamical system $\Phi$ if
  	\begin{equation}
  	\label{PM}
  	\mu_{s+\tau}=\mu_{s} \text{ and } \bar{\Theta}_{t}\mu_{s}=\mu_{t+s}, \text{ for all } t\geq 0, s\in \mathbb{R}.
  	\end{equation}
  	It is called a periodic measure with minimal period $\tau>0$ if $\tau$ is the smallest number such that (\ref{PM}) holds. It is an invariant measure if it also satisfies $\mu_{s}=\mu_{0}$ for any $s \in \mathbb{R}$, i.e. $\mu_{0}$ is an invariant measure of $\Phi$ if $\mu_{0}\in \mathcal{P}_{P}(\Omega \times \mathbb{X})$ and 
  	\begin{equation}
  	\bar{\Theta}_{t}\mu_{0}=\mu_{0}, \text{ for all } t \in \mathbb{R}^+.
  	\end{equation}
  \end{definition}
  
  \begin{theorem}
  	\label{RPPM}
  	(\cite{F.Z2}) If a random dynamical system $\Phi: \mathbb{R}^+\times \Omega \times \mathbb{X}\rightarrow \mathbb{X}$ has a random periodic path $Y: \mathbb{R}\times \Omega\rightarrow \mathbb{X}$, it has a periodic measure on $(\Omega\times \mathbb{X}, \mathcal{F}\otimes\mathcal{B}(\mathbb{X}))$, $\mu: \mathbb{R}\rightarrow \mathcal{P}_{P}(\Omega \times \mathbb{X})$, given by
  	\begin{equation}
  	\label{PMRDS}
  	\mu_{s}(A)=\int_{\Omega}\delta_{Y(s, \omega)}(A_{\theta_{s}\omega})P(d\omega),
  	\end{equation}
  	where $A_{\omega}$ is the $\omega$-section of $A$. Moreover, the time average of the periodic measure defined by
  	\begin{equation}
  	\bar{\mu}=\frac{1}{\tau}\int_{0}^{\tau}\mu_{s}ds
  	\end{equation}
  	is an invariant measure of $\Phi$ whose random factorisation is supported by $L^{\omega}$ defined in \eqref{path of pullback process}.
  	
  \end{theorem}
  
  \begin{remark}
  	For a periodic path Y, it is easy to see that the factorization of $\mu_{s}$ defined in Theorem \ref{RPPM} is
  	\begin{equation}
  	(\mu_{s})_{\omega}=\delta_{Y(s, \theta_{-s}\omega)}
  	\end{equation}
  	and
  	\begin{equation}
  	(\mu_{s+\tau})_{\omega}=(\mu_{s})_{\omega}, \quad \Phi(t, \omega)(\mu_{s})_{\omega}=(\mu_{t+s})_{\omega}.
  	\end{equation}
  \end{remark}

  In this section, we always assume the following condition and use the construction of periodic measure given in (\ref{PMRDS}).\\
  \\
  \textbf{Condition P}. \textit{There exists a random periodic path $Y$ with period $\tau$ for the random dynamical system $\Phi$}.
  
  Throughout the paper, we adopt the standard definition of ergodicity of a measure preserving dynamical system, i.e. any invariant set of the dynamical system has either full measure or zero measure.

  \subsection{The ergodicity of metric dynamical system on the Wiener space}

When we consider stochastic differential equations e.g. in Example \ref{exp2.2}, we need to consider metric dynamical systems on Wiener space, the shift of Brownian motion, of which the ergodicity under discrete time is one of the important conditions in our later set up. We establish the result here first.

It is well known that, for a canonical Wiener space $(\Omega, \mathcal{F}, P)$, the corresponding canonical dynamical system $(\Omega, \mathcal{F}, P, (\theta_t)_{t \geq 0})$ is ergodic, where $\theta_t$ is the Brownian shift. We will prove the discrete dynamical system $(\Omega, \mathcal{F}, P, (\theta_{\tau}^n)_{n \geq 0})$ is also ergodic. Thus our results in this paper can apply to stochastic differential equations and stochastic partial differential equations driven by Brownian motions.

\label{Example3}
A standard Brownian motion or Wiener process $(W_t)_{t \in \mathbb{T}}$ ($\mathbb{T}=\mathbb{R^+}$ (one-sided time) or $\mathbb{T}=\mathbb{R}$ (two-sided time)) in $\mathbb{R}^m$ is a process with $W_0=0$ and stationary independent increments satisfying $W_t-W_s \sim \mathcal{N}(0, |t-s|I)$. The corresponding measure $P$ on $(\Omega, \mathcal{F})$, where $\Omega=\mathcal{C}_0(\mathbb{T}, \mathbb{R}^m)$ and $\mathcal{F}$ is the Borel $\sigma$-algebra on $\Omega$, is called Wiener measure, the probability space $(\Omega, \mathcal{F}, P)$ is called Wiener space. The corresponding canonical metric dynamical system $\Sigma=(\Omega, \mathcal{F}, P, (\theta_t)_{t \in \mathbb{T}})$ describes Brownian motion or (Gaussian) white noise as a metric dynamical system of random dynamical systems generated by stochastic differential equations or stochastic partial differential equations driven by Brownian motions.

Let $\Sigma=(\Omega, \mathcal{F}, P, (\theta_t)_{t \in \mathbb{T}})$ be one of the canonical dynamical system introduced above, with the canonical filtration $\mathcal{F}_s^t:=\sigma(W_u-W_v, s\leq u,v \leq t), s\leq t.$ The following notations are standard (c.f. \cite{Arnold}),
$$\mathcal{T}^{\infty}:=\bigcap_{t \in \mathbb{T}}\mathcal{F}_t^{\infty},$$
and for two-sided time
$$\mathcal{T}_{-\infty}:=\bigcap_{t \in \mathbb{T}}\mathcal{F}_{-\infty}^t,$$
as the tail $\sigma$-algebras ($\mathcal{T}_{-\infty}$: remote past, $\mathcal{T}^{\infty}$: remote future). Set $$\mathcal{I}:=\{A\in \mathcal{F}: \theta_t^{-1}A=A, \text{ for all } t \in \mathbb{T}\},$$ 
and for a given $\tau>0$, define
$$\mathcal{I_{\tau}}:=\{A\in \mathcal{F}: \theta_{\tau}^{-1}A=A\}.$$
We say $\mathcal{A} \subset \mathcal{B}$ mod $P$ if for each $A \in \mathcal{A}$, there is a $B \in \mathcal{B}$ with $P(A\bigtriangleup B)=0$.

\begin{proposition}
	\label{nkey2}
	Assume that $\mathcal{I}, \mathcal{I_{\tau}}, \mathcal{T}^{\infty}$ and $\mathcal{T}_{-\infty}$ are defined as above. Then
	\begin{enumerate}[(i)]
		\item if $\mathbb{T}$ is one-sided ($\mathbb{T}=\mathbb{R^+}$): $\mathcal{I} \subset \mathcal{I_{\tau}} \subset \mathcal{T}^{\infty}$.
		\item if $\mathbb{T}$ is two-sided ($\mathbb{T}=\mathbb{R}$): $\mathcal{I} \subset \mathcal{I_{\tau}} \subset \mathcal{T}^{\infty}$ mod $P$, and $\mathcal{I} \subset \mathcal{T}_{-\infty}$ mod $P$.
	\end{enumerate}
\end{proposition}

\begin{proof}
	Obviously, by the definitions of $\mathcal{I}, \mathcal{I_{\tau}}$, we have $\mathcal{I} \subset \mathcal{I_{\tau}}$.
	
	$(i)$ We just need to prove that $\mathcal{I_{\tau}} \subset \mathcal{T}^{\infty}$. For any $A \in \mathcal{I_{\tau}}$, since 
	$$\mathcal{F}=\mathcal{B}(\Omega)=\sigma(W_t, t \in \mathbb{R^+}) =\sigma(W_u-W_v, 0\leq u, v <\infty)=\mathcal{F}_0^{\infty}$$
	and 
	$$\theta_{\tau}^{-1}\mathcal{F}_t^{\infty}=\mathcal{F}_{t+\tau}^{\infty} \text{ for all } t\in \mathbb{R^+},$$ 
	then 
	$$A=\theta_{\tau}^{-1}A \in \theta_{\tau}^{-1}\mathcal{F}_0^{\infty}=\mathcal{F}_{\tau}^{\infty}.$$ 
	By induction, we have $A\in \mathcal{F}_{n\tau}^{\infty}$ for all $n\in \mathbb{N}$. Then $A\in \bigcap_{n\in \mathbb{N}}\mathcal{F}_{n\tau}^{\infty}$. But 
	$$\mathcal{T}^{\infty}=\bigcap_{t\in \mathbb{R^+}}\mathcal{F}_t^{\infty}=\bigcap_{n\in \mathbb{N}}\mathcal{F}_{n\tau}^{\infty},$$
	thus $A\in \mathcal{T}^{\infty}$. This means $\mathcal{I_{\tau}}\subset \mathcal{T}^{\infty}$.\\
	
	$(ii)$ Without loss of generality, we just need to prove $\mathcal{I_{\tau}}\subset \mathcal{T}^{\infty}$ mod $P$. Similarly we can prove $\mathcal{I_{\tau}}\subset \mathcal{T}_{-\infty}$ mod $P$. To prove the desired result, for any $A\in \mathcal{I_{\tau}}$, we set
	$$A_n:=\{\omega\in \Omega: \text{ there exists } \omega' \in A \text{ such that } \omega(t)=\omega'(t) \text{ for all } t \in [-n\tau, \infty)\}$$
	for all $n \in \mathbb{N}$. We can conclude that $A_n\in \mathcal{F}_{-n\tau}^{\infty}$ and $A_n\supset A_{n+1} \supset A$. Assume $A_n\downarrow \tilde{A}$, then $A\subset \tilde{A}$ and $P(A)\leq P(\tilde{A})$. But since
	\begin{equation*}
		\begin{split}
		\mathcal{F}&=\sigma(W_t, t\in \mathbb{R})\\
		&=\{A|\text{ there exist } \{t_n\}_{n=1}^{\infty}\subset \mathbb{R}, B\in \mathcal{B}((\mathbb{R}^{m})^{\mathbb{N}}) \text{ such that } 
		A=\{\omega: \{\omega(t_n)\}_{n=1}^{\infty}\in B\} \},
		\end{split}
	\end{equation*}	
	so for a given $A\in \mathcal{I}_{\tau}$, there exist a sequence $\{t_n\}_{n=1}^{\infty}$ and a set $B\in \mathcal{B}((\mathbb{R}^{m})^{\mathbb{N}})$ such that $A=\{\omega: \{\omega(t_n)\}_{n=1}^{\infty}\in B\}$. Let 
	$$B^n=\{\Pi_nx=\{x_r\}_{r=1}^{n}: x=\{x_r\}_{r=1}^{\infty}\in B \}$$
	be the projection of $B$ from $(\mathbb{R}^{m})^{\mathbb{N}}$ to $(\mathbb{R}^m)^n$ and $$B_n=\{\omega:(\omega(t_1), \omega(t_2), \cdots, \omega(t_n))\in B^n\}.$$
	Then we know that $P(A)=\lim_{n\rightarrow \infty}P(B_n)$ from construction of finite dimensional distribution of Wiener measure. By the definition of $\tilde{A}$ we have $\tilde{A}\subset B_n$ for all $n$, then $P(\tilde{A})\leq \lim_{n\rightarrow \infty}P(B_n)=P(A)$. Since we know that $P(A)\leq P(\tilde{A})$, so we conclude $P(\tilde{A})=P(A)$ .
	
	Next we prove the following claim.\\
	Claim $(*)$: For all $n\geq 2$, we have $\theta_{\tau}^{-1}A_n \supset A_{n-1} \supset A_n$.
	
	Proof of Claim $(*)$: The claim $A_{n-1}\supset A_n$ is obvious. Now for any $\omega \in A_{n-1}$, there exists $\omega' \in A$ such that $\omega(t)=\omega'(t)$ for all $t \in [-(n-1)\tau, \infty)$, then $\theta_{\tau}\omega(s)=\omega(s+\tau)-\omega(\tau)=\omega'(s+\tau)-\omega'(\tau)=\theta_{\tau}\omega'(s)$ for all $s \in [-n\tau, \infty)$. Since $A$ is $\theta_{\tau}$-invariant set and $\omega'\in A$, so $\theta_{\tau}\omega' \in A$. Thus $\theta_{\tau}\omega \in A_n$ and $\omega \in \theta_{\tau}^{-1}A_n$. This means Claim $(*)$ holds.\\
	
	Now we continue our proof. Since $\theta_{\tau}$ preserves probability $P$, then by Claim $(*)$, we have
	$$P(A_n)=P(\theta_{\tau}^{-1}A_n)\geq P(A_{n-1})\geq P(A_n).$$
	It turns out that
	$$P(A_n)=P(\tilde{A})=P(A) \text{ for all } n\in \mathbb{N}.$$
	Let $B_k=\theta_{k\tau}^{-1}A_2$. By Claim $(*)$, $\theta_{\tau}^{-1}A_2\supset A_2,$ i.e. $\theta_{\tau}A_2 \subset A_2$. So for any $\omega \in B_k=\theta_{k\tau}^{-1}A_2$, then $\theta_{k\tau}\omega \in A_2$, thus $\theta_{\tau}(\theta_{k\tau}\omega)=\theta_{(k+1)\tau}\omega \in A_2$, therefore $\omega \in \theta_{(k+1)\tau}^{-1}A_2=B_{k+1}$. This means $B_k\subset B_{k+1}$. And also we have that $$P(B_k)=P(\theta_{k\tau}^{-1}A_2)=P(A_2)=P(A_n)=P(\tilde{A})=P(A),$$ 
	for all $k, n \in \mathbb{N}$. Since $A_2\in \mathcal{F}_{-2\tau}^{\infty}$, then $B_k=\theta_{k\tau}^{-1}A_2 \in \mathcal{F}_{(-2+k)\tau}^{\infty}$. Let 
	$$\tilde{B}=\lim_{k\rightarrow \infty}B_k=\bigcup_{k\in \mathbb{N}}B_k.$$
	Then
	$$\tilde{B}\in \bigcap_{k\in \mathbb{N}}\mathcal{F}_{(-2+k)\tau}^{\infty}=\mathcal{T}^{\infty}$$
	and
	$$\tilde{B}\supset B_k \supset A_n \supset \tilde{A} \supset A$$
	and
	$$P(\tilde{B})=P(B_k)=P(A_n)=P(\tilde{A})=P(A),$$
	for all $k,n \in \mathbb{N}$.
	
	Then for any $A\in \mathcal{I_{\tau}}$, there exist $\tilde{B}\in \mathcal{T}^{\infty}$ such that $P(A\bigtriangleup \tilde{B})=P(\tilde{B}\setminus A)=P(\tilde{B})-P(A)=0$, i.e. $\mathcal{I_{\tau}} \subset \mathcal{T}^{\infty}$ mod $P$.
\end{proof}

\begin{theorem}
	\label{lastTH}
	The canonical dynamical systems driven by Brownian motion $\Sigma=(\Omega, \mathcal{F}, P,\\
	 (\theta_t)_{t \in \mathbb{T}})$ $(\mathbb{T}=\mathbb{R^+}$ or $\mathbb{R})$ and their discrete dynamical systems $\Sigma^{\tau}=(\Omega, \mathcal{F}, P, (\theta_{\tau}^n)_{n\geq 0})$ are ergodic.
\end{theorem}
\begin{proof}
	By Proposition \ref{nkey2}, if $\mathcal{T}^{\infty}$ is trival mod $P$, $\Sigma$ and $\Sigma^{\tau}$ are ergodic. Since these canonical dynamical systems $\Sigma, \Sigma^{\tau}$ are driven by a standard Brownian motion, then the tail $\sigma$-algebra $\mathcal{T}^{\infty}$ is trival mod $P$ by Kolmogorov's zero-one law.
\end{proof}

\begin{remark}
	(i) It is noted that the ergodicity of $\Sigma$ in Theorem \ref{lastTH} was known in literature (c.f. \cite{Arnold}). The main purpose of the Theorem is to prove $\Sigma^{\tau}$ is ergodic in the discrete case. This result is new. But the continuous case $\Sigma$ being ergodic is proved as a byproduct of the techniques we build here.
	
	(ii) Proposition \ref{nkey2} and Theorem \ref{lastTH} hold also for Wiener process on a separable Hilbert space where the Wiener measure was given in \cite{Da Prato1}. The proof is exactly the same.
	
	(iii) For a given two-sided dynamical system $(\Omega, \mathcal{F}, P, (\theta_t)_{t\in \mathbb{R}})$, we know that the transformation $\theta_t: \Omega\to \Omega$ is invertible and $\theta_t^{-1}=\theta_{-t}$ for each $t\in \mathbb{R}$. Then it is easy to verify that $\theta_t^{-1}A=A$ for all $t\in \mathbb{R}$ if and only if $\theta_t^{-1}A=A$ for all $t\geq 0$, which indicates that ``$(\Omega, \mathcal{F}, P, (\theta_t)_{t\in \mathbb{R}})$ being ergodic" is equivalent to ``$(\Omega, \mathcal{F}, P, (\theta_t)_{t\geq 0})$ being ergodic". In the following, we will always only assume that $(\Omega, \mathcal{F}, P, (\theta_t)_{t\geq 0})$ is ergodic, which actually also indicates the ergodicity of the dynamical system $(\Omega, \mathcal{F}, P, (\theta_t)_{t\in \mathbb{R}})$.
\end{remark}

  \subsection{Ergodicity of skew product dynamical systems in the case of periodic measures}
  \begin{theorem}
  	\label{ET}
  	Assume Condition P. If the metric dynamical system $(\Omega, \mathcal{F}, P, (\theta_{\tau}^{n})_{n\geq 0})$ is ergodic, then the skew product dynamical systems $(\bar{\Omega}, \bar{\mathcal{F}}, \mu_{s}, (\bar{\Theta}_{\tau}^{n})_{n\geq 0})$ for each $s\in \mathbb{R}$ and $(\bar{\Omega}, \bar{\mathcal{F}}, \bar{\mu}, \\
  	(\bar{\Theta}_{t})_{t\geq 0})$ are ergodic.
  \end{theorem}
  \begin{proof}
  	First, note by Theorem \ref{RPPM},  that $\mu_{s}$ defined in (\ref{PMRDS})  is a periodic measure on $(\bar{\Omega}, \bar{\mathcal{F}})$, so $\bar{\Theta}_{\tau}$ preserves measures $\mu_{s}$ for each $s\in \mathbb{R}$ and $(\bar{\Theta}_{t})_{t\geq 0}$ preserves the measure $\bar{\mu}$. 
  	
  	Next, we will show that the dynamical system $(\bar{\Omega}, \bar{\mathcal{F}}, \mu_{s}, (\bar{\Theta}_{\tau}^n)_{n\geq 0})$ is ergodic for any fixed $s\in \mathbb{R}$. By definition of ergodicity, we need to show that for any $A\in \bar{\mathcal{F}}$ with $\bar{\Theta}_{\tau}^{-1}A=A$, either $\mu_{s}(A)=0$ or 1. Define $A_{s}:=\{\omega: (\omega, Y(s,\theta_{-s}\omega))\in A\}$, then
  	\begin{equation}
  	\label{key1}
  	\begin{split}
  	\mu_{s}(A)
  	&=\int_{\Omega}\I_{A_{\omega}}(Y(s, \theta_{-s}\omega))P(d\omega)\\
  	&=P(\{\omega: (\omega, Y(s, \theta_{-s}\omega))\in A\})\\
  	&=P(A_{s}).\\
  	\end{split}
  	\end{equation}
  	Moreover
  	\begin{equation*}
  	\begin{split}
  	\theta_{\tau}^{-1}A_{s}&=\{\omega: \theta_{\tau}\omega\in A_{s}\}\\
  	&=\{\omega: (\theta_{\tau}\omega, Y(s, \theta_{-s}\theta_{\tau}\omega))\in A\}\\
  	&=\{\omega: \bar{\Theta}_{\tau}(\omega, Y(s-\tau, \theta_{-(s-\tau)}\omega))\in A\}\\
  	&=\{\omega: (\omega, Y(s-\tau, \theta_{-(s-\tau)}\omega))\in \bar{\Theta}_{\tau}^{-1}A\}\\
  	&=\{\omega: (\omega, Y(s-\tau, \theta_{-(s-\tau)}\omega))\in A\}\\
  	&=\{\omega: (\omega, Y(s, \theta_{-s}\omega))\in A\}\\
  	&=A_{s},\\
  	\end{split}
  	\end{equation*}
  	for all $s\in \mathbb{R}$. Thus $A_s$ is an invariant set with respect to $\theta_{\tau}$.
  	
  	Since $(\Omega, \mathcal{F}, P, (\theta_{\tau}^{n})_{n\geq 0})$ is ergodic, then we have $P(A_{s})=0$ or $P(A_{s})=1$. Thus $\mu_{s}(A)=P(A_s)=0$ or 1 from (\ref{key1}).
  	
  	Let us show that the dynamical system $(\bar{\Omega}, \bar{\mathcal{F}}, \bar{\mu}, (\bar{\Theta}_t)_{t \geq 0})$ is ergodic, i.e. for any $A\in \bar{\mathcal{F}}$ with $\bar{\Theta}_{t}^{-1}A=A$ for all $t \in \mathbb{R}^+$, we need to prove that $\bar{\mu}(A)=0$ or 1. For such $A$, by what we just proved, we know that $\mu_{s}(A)=0$ or 1 for all $s\in \mathbb{R}$. On the other hand, since
  	\begin{equation}
  	\begin{split}
  	\theta_{t}^{-1}A_{s}&=\{\omega: \theta_{t}\omega\in A_{s}\}\\
  	&=\{\omega: (\theta_{t}\omega, Y(s, \theta_{-s}\theta_{t}\omega))\in A\}\\
  	&=\{\omega: \bar{\Theta}_{t}(\omega, Y(s-t, \theta_{-(s-t)}\omega))\in A\}\\
  	&=\{\omega: (\omega, Y(s-t, \theta_{-(s-t)}\omega))\in \bar{\Theta}_{t}^{-1}A\}\\
  	&=\{\omega: (\omega, Y(s-t, \theta_{-(s-t)}\omega))\in A\}\\
  	&=A_{s-t},\\
  	\end{split}
  	\end{equation}
  	 we have $\mu_{s}(A)=P(A_{s})=P(\theta_{t}^{-1}A_{s})=P(A_{s-t})=\mu_{s-t}(A)$. This means $\mu_{s}(A)=\mu_{0}(A)$. Hence $\bar{\mu}(A)=\frac{1}{\tau}\int_{0}^{\tau}\mu_{s}(A)ds=\mu_{0}(A)=0$ or 1.
  \end{proof}
The following theorem shows that the converse of Theorem \ref{ET} also holds.
\begin{theorem}
	\label{Converse ET}
	Assume Condition P. If $(\bar{\Omega}, \bar{\mathcal{F}}, \mu_{s}, (\bar{\Theta}_{\tau}^n)_{n\geq 0})$ is ergodic for some $s\in \mathbb{R}$, then $(\Omega, \mathcal{F}, P, (\theta_{\tau}^n)_{n\geq 0})$ is ergodic.
\end{theorem}

\begin{proof}
	Fix an $F\in \mathcal{F}$ with $\theta_{\tau}^{-1}F=F$. Let $\bar{F}=F\times \mathbb{X} \in \bar{\mathcal{F}}$, then we have
	\begin{equation}
	 \label{key4}
		\begin{split}
		\bar{\Theta}_{\tau}^{-1}\bar{F}&=\{(\omega, x): \bar{\Theta}_{\tau}(\omega, x)\in \bar{F}\}\\
		     &=\{(\omega, x): (\theta_{\tau}\omega, \Phi(\tau, \omega)x)\in F\times \mathbb{X}\}\\
		     &=\{(\omega, x): \theta_{\tau}\omega \in F\}\\
		     &=(\theta_{\tau}^{-1}F)\times \mathbb{X} = F\times \mathbb{X}=\bar{F}.\\
		\end{split}
	\end{equation}
	And by (\ref{key1}), we also have
	$$\mu_{s}(\bar{F})=P(\bar{F}_s) \text{ for all } s\in \mathbb{R},$$
	where $\bar{F}_s=\{\omega: (\omega, Y(s, \theta_{-s}\omega))\in \bar{F}\}=\{\omega: (\omega, Y(s, \theta_{-s}\omega))\in F\times \mathbb{X}\}=F.$
	Thus
	\begin{equation}
	\label{key2}
	\mu_{s}(\bar{F})=P(F) \text{ for all } s\in \mathbb{R}.
	\end{equation}
    Since $(\bar{\Omega}, \bar{\mathcal{F}}, \mu_{s}, (\bar{\Theta}_{\tau}^n)_{n\geq 1})$ for some $s\in \mathbb{R}$ is ergodic, by (\ref{key4}) and (\ref{key2}) we have 
    $$P(F)=\mu_{s}(\bar{F})=0 \text{ or } 1.$$
\end{proof}

\begin{remark}
	\label{re1}
	(i). From Theorem \ref{ET} and Theorem \ref{Converse ET}, we can conclude that $(\bar{\Omega}, \bar{\mathcal{F}}, \mu_{s}, (\bar{\Theta}_{\tau}^n)_{n\geq 0})$ is ergodic for some $s\in \mathbb{R}$ implies that $(\bar{\Omega}, \bar{\mathcal{F}}, \mu_{s}, (\bar{\Theta}_{\tau}^n)_{n\geq 0})$ is ergodic for all $s\in \mathbb{R}$. This conclusion does not seem to be true in the phase space case, which we will consider in the next section.

	(ii). It is easy to check that if the dynamical system $(\Omega, \mathcal{F}, P, (\theta_{\tau}^n)_{n \geq 0})$ is ergodic, then $(\Omega, \mathcal{F}, P, (\theta_t)_{t \geq 0})$ will be ergodic. The converse is not true in general. A counterexample is given below.
	
\end{remark}

\begin{example} (A metric dynamical system on torus)
\label{Example2}
Now we consider $\tilde{\Omega}=[0, 1)\times [0, 1)$ and $\tilde{\mathcal{F}}:=\mathcal{B}(\tilde{\Omega})$. Define the map $\tilde{\theta}^{\alpha}: \mathbb{R}\times \tilde{\Omega} \rightarrow \tilde{\Omega}$ by
$$\tilde{\theta}^{\alpha}_t(r, x):=((r+t) \ \text{mod } 1, (x+t\alpha)\  \text{mod } 1) \text{ for all } t\in \mathbb{R}, (r, x)\in \tilde{\Omega},$$
where $\alpha$ is a fixed positive irrational number. Then it is easy to check that $(\tilde{\Omega}, \tilde{\mathcal{F}}, (\tilde{\theta}^{\alpha}_t)_{t\in \mathbb{R}})$ is a dynamical system. Let $L$ be the Lebesgue measure on $[0, 1)$ and $\mathcal{P}(\tilde{\Omega})$ be the set of probability measures on $(\tilde{\Omega}, \tilde{\mathcal{F}})$. Define $\mu: \mathbb{R} \rightarrow \mathcal{P}(\tilde{\Omega})$ by 
$$\mu_{s}:= \delta_{\{s\ \text{mod }1\}}\times L,$$
then $\mu$ is a periodic probability measure with period 1 on $(\tilde{\Omega}, \tilde{\mathcal{F}})$. Let $\bar{\mu}:= \int_{0}^{1}\mu_{s}ds$.
\end{example}
\begin{proposition}
	\label{prop1}
	The dynamical systems $(\tilde{\Omega}, \tilde{\mathcal{F}}, \mu_{s}, ((\tilde{\theta}^{\alpha}_1)^n)_{n\geq 0}) \text{ for all } s\in \mathbb{R}$ and $(\tilde{\Omega}, \tilde{\mathcal{F}}, \bar{\mu}, \\ 
	(\tilde{\theta}^{\alpha}_t)_{t\geq 0})$ are ergodic.
\end{proposition}

\begin{proof}
	Let $\Omega:=[0,1), \mathcal{F}:=\mathcal{B}(\Omega)$ and $\theta_{\alpha}: \Omega \rightarrow \Omega, \theta_{\alpha}(x)=(x+\alpha)\ \text{mod } 1$. It is well known that the dynamical system $(\Omega, \mathcal{F}, L, (\theta_{\alpha}^n)_{n \in \mathbb{N}})$ with an irrational number $\alpha$ is ergodic.
	
	Fix an $\tilde{A}\in \tilde{\mathcal{F}}$ with $(\tilde{\theta}^{\alpha}_1)^{-1}\tilde{A}=\tilde{A}$. Define $\tilde{A}_{r}:=\{x: (r, x)\in \tilde{A}\}$ for all $r\in [0, 1)$, then for any $s\in \mathbb{R}$
	\begin{equation}
	\label{key5}
	\mu_{s}(\tilde{A})=\int_{\tilde{\Omega}}I_{\tilde{A}}(r,x)\delta_{\{s\ \text{mod }1\}}(dr)L(dx)
	=\int_{0}^{1}I_{\tilde{A}}(s\ \text{mod }1, x)L(dx)
	=L(\tilde{A}_{(s\ \text{mod }1)}),
	\end{equation}
	and
	\begin{equation}
	\label{key6}
	\begin{split}
	\theta_{\alpha}^{-1}\tilde{A}_{(s\ \text{mod }1)}&=\{x: \theta_{\alpha}x\in \tilde{A}_{(s\ \text{mod }1)}\}\\
	&=\{x: (s\ \text{mod 1}, (x+\alpha)\ \text{mod }1)\in \tilde{A}\}\\
	&=\{x: \tilde{\theta}^{\alpha}_1(s\ \text{mod }1, x)\in \tilde{A}\}\\
	&=\{x: (s\ \text{mod }1, x)\in (\tilde{\theta}^{\alpha}_1)^{-1}\tilde{A}=\tilde{A}\}\\
	&=\tilde{A}_{(s\ \text{mod }1)}.
	\end{split}
	\end{equation}
	Applying the ergodicity of the dynamical system $(\Omega, \mathcal{F}, L, (\theta_{\alpha}^n)_{n\geq 0})$ and (\ref{key5}), (\ref{key6}), we have
	$$\mu_{s}(\tilde{A})=L(\tilde{A}_{(s\ \text{mod }1)})=0 \text{ or } 1.$$
	This means the dynamical systems $(\tilde{\Omega}, \tilde{\mathcal{F}}, \mu_{s}, ((\tilde{\theta}^{\alpha}_1)^n)_{n\geq 0})$ for all $s\in \mathbb{R}$ are ergodic.
	
	Again for any $\tilde{A} \in \tilde{\mathcal{F}}$ with $(\tilde{\theta}^{\alpha}_t)^{-1}\tilde{A}=\tilde{A}$ for all $t\in \mathbb{R}$, then $(\tilde{\theta}^{\alpha}_1)^{-1}\tilde{A}=\tilde{A}$ and $\mu_0(\tilde{A})=0$ or 1. Since
	$$\tilde{A}=\bigcup_{s\in [0,1)}(\{s\}\times \tilde{A}_s)$$
	and
	\begin{equation*}
	\begin{split}
	(\tilde{\theta}^{\alpha}_t)^{-1}\tilde{A}
	&=\bigcup_{r\in [0,1)}(\tilde{\theta}^{\alpha}_t)^{-1}(\{r\}\times \tilde{A}_r)\\
	&=\bigcup_{r\in [0,1)}\{(s, x): s\in [0,1), \tilde{\theta}^{\alpha}_t(s,x)\in \{r\}\times \tilde{A}_r\}\\
	&=\bigcup_{r\in [0,1)}\{(s, x): s\in [0,1), ((s+t)\ \text{mod }1, (x+t\alpha)\ \text{mod }1)\in \{r\}\times \tilde{A}_r\}\\
	&=\bigcup_{r\in [0,1)}\{(s, x): s\in [0,1), (s+t)\ \text{mod }1=r, \theta_{t\alpha}x\in \tilde{A}_r\}\\
	&=\bigcup_{r\in [0,1)}\{(s, x): s\in [0,1), (s+t)\ \text{mod }1=r, \theta_{t\alpha}x\in \tilde{A}_{((s+t)\ \text{mod }1)}\}\\
	&=\bigcup_{s\in [0,1)}\{(s, x): \theta_{t\alpha}x\in \tilde{A}_{((s+t)\ \text{mod }1)}\}\\
	&=\bigcup_{s\in [0,1)}(\{s\}\times \theta_{t\alpha}^{-1}\tilde{A}_{((s+t)\ \text{mod }1)}),\\
	\end{split}
	\end{equation*}
	so $\bigcup_{s\in [0,1)}(\{s\}\times \theta_{t\alpha}^{-1}\tilde{A}_{((s+t) \ \text{mod }1)})=\bigcup_{s\in [0,1)}(\{s\}\times \tilde{A}_s)$ and hence $\theta_{t\alpha}^{-1}\tilde{A}_{((s+t) \ \text{mod }1)}=\tilde{A}_s$ for all $t\in \mathbb{R}, s\in [0, 1)$. In particular, $\theta_{t\alpha}^{-1}\tilde{A}_t=\tilde{A}_0$ for all $t\in [0, 1)$. Therefore
	\begin{equation}
	\label{eq4}
	\mu_{s}(\tilde{A})=L(\tilde{A}_s)=L(\theta_{s\alpha}^{-1}\tilde{A}_s)=L(\tilde{A}_0)=\mu_0(\tilde{A})
	\text{ for all } s\in [0,1).	
	\end{equation}
	It turns out that $\bar{\mu}(\tilde{A})=\int_{0}^{1}\mu_{s}(\tilde{A})ds=\mu_0(\tilde{A})=0$ or 1, which means the dynamical system $(\tilde{\Omega}, \tilde{\mathcal{F}}, \bar{\mu}, (\tilde{\theta}^{\alpha}_t)_{t\geq 0})$ is ergodic.
\end{proof}

\begin{remark}
	It is easy to see that the dynamical system $(\tilde{\Omega}, \tilde{\mathcal{F}}, \bar{\mu}, ((\tilde{\theta}^{\alpha}_1)^n)_{n\geq 0})$ is not ergodic considering the first coordinate of the mapping. This provides an example that the continuous time dynamical system is ergodic, but its discretization may not be.
\end{remark}

\subsection{Ergodicity of skew product dynamical systems in the case of invariant measures}
The following result is well-known:
\begin{theorem}(\cite{Arnold})
	\label{SIM}
	If a random dynamical system $\Phi: \mathbb{R}^+\times \Omega\times \mathbb{X} \rightarrow \mathbb{X}$ has a stationary path $Y: \Omega\rightarrow \mathbb{X}$, it has an invariant measure on $(\bar{\Omega}, \bar{\mathcal{F}})$, $\mu \in \mathcal{P}_{P}(\Omega\times \mathbb{X})$ defined by
	\begin{equation}
	\label{IMmu}
		\mu(A)=\int_{\Omega}\delta_{Y(\omega)}(A_{\omega})P(d\omega),
	\end{equation} 
	where $A_{\omega}$ is the $\omega$-section of $A$ and $\mu_{\omega}=\delta_{Y(\omega)}$.
\end{theorem}

\begin{theorem}
	\label{SPET}
	Assume that the random dynamical system $\Phi$ has a stationary path $Y$. Then the metric dynamical system $(\Omega, \mathcal{F}, P, (\theta_{t})_{t\geq 0})$ is ergodic if and only if the skew product dynamical system $(\bar{\Omega}, \bar{\mathcal{F}}, \mu, (\bar{\Theta}_t)_{t\geq 0})$ is ergodic, where $\mu$ is defined by (\ref{IMmu}).
\end{theorem}

\begin{proof}
	First prove the ``$\Rightarrow$" part. Assume $(\Omega, \mathcal{F}, P, (\theta_{t})_{t\geq 0})$ is ergodic.
	Since $\mu$ is an invariant measure on $(\bar{\Omega}, \bar{\mathcal{F}})$, so $\bar{\Theta}_{t}$ preserves the measure $\mu$. To prove $(\bar{\Omega}, \bar{\mathcal{F}}, \mu, (\bar{\Theta}_t)_{t\geq 0})$ is ergodic, we need to prove that for any $A\in \bar{\mathcal{F}}$ with $\bar{\Theta}_{t}^{-1}A=A \text{ for all } t\geq 0$, $\mu(A)=0$ or 1. Set $A_{0}:=\{\omega: (\omega, Y(\omega))\in A\}$, then for any $t \geq 0$, we have
	\begin{equation*}
	\begin{split}
	\theta_{t}^{-1}A_{0}&=\{\omega: \theta_{t}\omega \in A_0\}\\
	                             &=\{\omega: (\theta_{t}\omega, Y(\theta_{t}\omega))\in A\}\\
	                             &=\{\omega: \bar{\Theta}_{t}(\omega, Y(\omega))\in A\}\\
	                             &=\{\omega: (\omega, Y(\omega))\in \bar{\Theta}_{t}^{-1}A\}\\
	                             &=\{\omega: (\omega, Y(\omega))\in A\}\\
	                             &=A_{0}.\\
	\end{split}
	\end{equation*}
	This means that $A_{0}$ is an invariant set with respect to $(\theta_{t})_{t\geq 0}$. Note that
	$$\mu(A)=\int_{\Omega}\delta_{Y(\omega)}(A_{\omega})P(d\omega)=P(\{\omega:(\omega, Y(\omega))\in A\})=P(A_0).$$
	By the ergodicity of $(\Omega, \mathcal{F}, P, (\theta_{t})_{t\geq 0})$, we have $P(A_{0})=0$ or 1. Therefore $\mu(A)=0$ or 1.\\
	
	Now we prove the ``$\Leftarrow$" part. Assume $(\bar{\Omega}, \bar{\mathcal{F}}, \mu, (\bar{\Theta}_t)_{t\geq 0})$ is ergodic. For any $F\in \mathcal{F}$ with $\theta_{t}^{-1}F=F$ for all $t\geq 0$, we consider $\bar{F}=F\times\mathbb{X}\in \bar{\mathcal{F}}$, then for any $t\geq 0$
	\begin{equation*}
		\begin{split}
		\bar{\Theta}_{t}^{-1}\bar{F}&=\{(\omega, x): \bar{\Theta}_{t}(\omega, x)\in \bar{F}\}\\
		&=\{(\omega, x): (\theta_{t}\omega, \Phi(t, \omega)x)\in F\times \mathbb{X}\}\\
		&=\{(\omega, x): \theta_{t}\omega \in F\}\\
		&=(\theta_{t}^{-1}F)\times \mathbb{X} = F\times \mathbb{X}=\bar{F}.\\
		\end{split}
	\end{equation*}
	Since $$\mu(\bar{F})=P(\{\omega: (\omega, Y(\omega))\in F\times \mathbb{X}\})=P(F),$$
	by the ergodicity of $(\bar{\Omega}, \bar{\mathcal{F}}, \mu, (\bar{\Theta})_{t\geq 0})$, we have
	$$P(F)=\mu(\bar{F})=0 \text{ or }1.$$
\end{proof}

\section{Ergodicity of canonical Markovian systems from periodic measures: necessary and sufficient conditions}
\subsection{ Ergodicity of periodic measures in Markovian setting}

Now we consider a Markovian cocycle random dynamical system $\Phi$ on a filtered dynamical system $(\Omega, \mathcal{F}, P, (\theta_{t})_{t \in \mathbb{R}}, (\mathcal{F}_s^t)_{s \leq t})$, i.e. $\mathcal{F}_s^t\subset \mathcal{F}$, assuming for any $s,t, u\in \mathbb{R}, s\leq t,$ $\theta_{u}^{-1}\mathcal{F}_s^t=\mathcal{F}_{s+u}^{t+u}$ and for any $t\in \mathbb{R}^+, \Phi(t,\cdot)$ is measurable with respect to $\mathcal{F}_0^t$. We also assume the random periodic path $Y(s)$ is adapted, that is to say that for each $s\in \mathbb{R}, Y(s,\cdot)$ is measurable with respect to $\mathcal{F}_{-\infty}^s:=\vee_{r\leq s}\mathcal{F}_r^s$.

Denote the transition probability of Markovian process $\Phi(t,\cdot)x$ on the polish space $\mathbb{X}$ with Borel $\sigma$-field $\mathcal{B}(\mathbb{X})$ by (c.f. \cite{Arnold}, \cite{Da Prato})
$$P_t(x,\Gamma)=P(\{\omega: \Phi(t, \omega)x\in \Gamma\}), \quad t\in \mathbb{R}^+, \quad \Gamma\in \mathcal{B}(\mathbb{X}).$$

Denote by $L_{b}(\mathbb{X})$ the set of all real-valued bounded Borel measurable functions defined on $\mathbb{X}$ and $\mathcal{P}(\mathbb{X})$ be the set of all probability measures defined on $(\mathbb{X}, \mathcal{B}(\mathbb{X}))$. For any $t \geq 0$ and $\rho \in \mathcal{P}(\mathbb{X})$ we set 
$$P_{t}^*\rho(\Gamma)=\int_{\mathbb{X}}P_{t}(x, \Gamma)\rho(dx), \quad \Gamma \in \mathcal{B}(\mathbb{X}),$$
and for any $\varphi \in L_b(\mathbb{X})$, define
\begin{equation}
   \label{Markov semigroup}
	(P_t\varphi)(x)=\int_{\mathbb{X}}P_t(x, dy)\varphi(y), \quad x\in \mathbb{X},
\end{equation}
as a semigroup from $L_b(\mathbb{X})$ to $L_b(\mathbb{X})$.

\begin{definition}
	(\cite{F.Z2}) A measure function $\rho.: \mathbb{R}\rightarrow \mathcal{P}(\mathbb{X})$ is called a periodic measure of period $\tau$ on $(\mathbb{X}, \mathcal{B}(\mathbb{X}))$ for the Markovian semigroup $P_{t}$ if it satisfies
	\begin{equation}
	\label{rho}
		P_{t}^*\rho_{s}=\rho_{s+t} \text{ and } \rho_{s+\tau}=\rho_{s}, \text{ for all } s\in \mathbb{R}, t\in \mathbb{R}^+.
	\end{equation}
	It is called a periodic measure with minimal period $\tau$ if $\tau >0$ is the smallest number such that (\ref{rho}) holds. It is called an invariant measure if it satisfies $\rho_{s}=\rho_{0}$ for all $s\in \mathbb{R}, i.e.$ $ \rho_{0}$ is an invariant measure for the Markovian semigroup $P_{t}$ if
	\begin{equation}
		P_{t}^*\rho_{0}=\rho_{0}, \text{ for all } t\in \mathbb{R}^+.
	\end{equation}
\end{definition}

With a given Markovian semigroup $P_t, t\geq 0$, and an invariant measure $\rho\in \mathcal{P}(\mathbb{X})$, we will associate now, in the following unique way, a dynamical system $(\Omega^*, \mathcal{F}^*, (\theta^*_t)_{t\in \mathbb{R}}, \mathbb{P}^{\rho})$ on the space $\Omega^*=\mathbb{X}^{\mathbb{R}}$ of all $\mathbb{X}$-valued functions.

Define $\Omega^*=\mathbb{X}^\mathbb{R}$, the space of all $\mathbb{X}$-valued functions on $\mathbb{R}$, $\mathcal{F}^*$ is the smallest $\sigma$-algebra containing all cylindrical sets of $\Omega^*$. And the shift $\theta^*: \mathbb{R}\times \Omega^*\rightarrow \Omega^*$ defined by $(\theta_{t}^*\omega^*)(s)=\omega^*(t+s)$, for all $\omega^*\in \Omega^*.$ For an arbitrary invariant measure $\rho \in \mathcal{P}(\mathbb{X})$ and an arbitrary finite set $I=\{t_1,t_2,\cdots,t_n\}, t_1<t_2<\cdots<t_n,$ we can define a probability $\mathbb{P}^{\rho}_I$ on $(\mathbb{X}^I, \mathcal{B}(\mathbb{X}^I))$ by the formula
$$\mathbb{P}_I^{\rho}(\Gamma)= \int_{\mathbb{X}}\rho(dx_1)\int_{\mathbb{X}}P_{t_2-t_1}(x_1,dx_2)\cdots  \int_{\mathbb{X}}P_{t_n-t_{n-1}}(x_{n-1},dx_n) I_{\Gamma}(x_1,x_2,\cdots,x_n), \quad \Gamma\in \mathcal{B}(\mathbb{X}^I).$$
By the Kolmogorov extension theorem, there exists a unique probability measure $\mathbb{P}^{\rho}$ on $(\Omega^*, \mathcal{F}^*)$ such that for every finite set $I=\{t_1,t_2,\cdots,t_n\}$ and $\Gamma\in \mathcal{B}(\mathbb{X}^I)$
$$\mathbb{P}^{\rho}(\{\omega^*:(\omega^*(t_1), \omega^*(t_2), \cdots, \omega^*(t_n))\in \Gamma\}) = \mathbb{P}^{\rho}_I(\Gamma).$$
From \cite{Da Prato}, the transformations $\theta^*_t, t\in \mathbb{R}$ preserve the measure $\mathbb{P}^{\rho}$, and the quadruplet $S^{\rho}=(\Omega^*, \mathcal{F}^*, (\theta^*_t)_{t\in \mathbb{R}}, \mathbb{P}^{\rho})$ defines a dynamical system, called the \textit{canonical dynamical system} associated with $P_t, t\geq 0, \rho$ and $\theta^*_t$.
\begin{definition}
	(\cite{Da Prato}) The invariant measure $\rho$ is said to be ergodic with respect to the Markovian semigroup $P_t, t\geq 0$, if its associated canonical dynamical system $S^{\rho}=(\Omega^*, \mathcal{F}^*, (\theta^*_t)_{t\geq 0},
	 \mathbb{P}^{\rho})$ is ergodic.
\end{definition}

\begin{definition}
	(\cite{F.Z2}) The $\tau$-periodic measure $\{\rho_{s}\}_{s\in \mathbb{R}}$ is said to be PS-ergodic if for each $s\in [0, \tau), \rho_s$ as the invariant measure of the $\tau$-mesh discrete Markovian semigroup $\{P_{k\tau}\}_{k\in \mathbb{N}}$, at integral multiples of the period on the Poincar\'{e} section, is ergodic.
\end{definition}

We also recall the following theorem proved in \cite{F.Z2}.
\begin{theorem}
	\label{PM2}
	Assume the Markovian cocycle $\Phi: \mathbb{R}^+\times \Omega\times \mathbb{X} \rightarrow \mathbb{X}$ has an adapted random periodic path $Y: \mathbb{R}\times \Omega\rightarrow \mathbb{X}$. Then the measure function $\rho.: \mathbb{R}\rightarrow \mathcal{P}(\mathbb{X})$ defined by
	\begin{equation}
	\label{nkey1}
		\rho_{s}:=E_P(\mu_{s}).=E_P\delta_{Y(s, \theta_{-s}\cdot)}=E_P\delta_{Y(s, \cdot)}, \text{ for all } s\in \mathbb{R},
	\end{equation} 
	which is the law of the random periodic path $Y$, is a periodic measure of the semigroup $P_t$ on $(\mathbb{X}, \mathcal{B}(\mathbb{X}))$. Its time average $\bar{\rho}$ over a time interval of exactly one period defined by 
	\begin{equation}
	\label{time average measure rho}
		\bar{\rho}=\frac{1}{\tau}\int_{0}^{\tau}\rho_sds,
	\end{equation}
	is an invariant measure and satisfies that for any $\Gamma \in \mathcal{B}(\mathbb{X}), t\in \mathbb{R},$
	\begin{equation*}
	\begin{split}
	\bar{\rho}(\Gamma)&=E_P\left[\frac{1}{\tau}\{s\in [0, \tau): Y(s,\cdot)\in \Gamma\}\right]\\
	                             &=E_P\left[\frac{1}{\tau}\{s\in [t, t+\tau): Y(s,\cdot)\in \Gamma\}\right].\\
	\end{split}
	\end{equation*}
\end{theorem}

\subsection{Ergodic canonical dynamical systems generated from periodic measure}
Given a random periodic path $Y$ of the Markovian cocycle $\Phi$, for any given $s\in \mathbb{R}$, define
$$L_s^{\omega}:=\{Y(s+k\tau, \omega): k\in \mathbb{Z}\},$$
and the $\rho_s$-invariant set $\mathcal{I}_s^{\tau}$ associated with the discrete Markovian semigroup $(P_{\tau}^k)_{k\in \mathbb{N}},$  
$$\mathcal{I}_s^{\tau}:=\{\Gamma\in \mathcal{B}(\mathbb{X}): P_{\tau}I_{\Gamma}=I_{\Gamma}, \rho_{s}-a.s.\}.$$
For these two sets, we consider\\
\textbf{Condition A.}  \textit{For any $s\in \mathbb{R}$, $\Gamma\in \mathcal{I}_s^{\tau}$, one has for $P$-almost all $\omega\in \Omega$, either $L_s^{\omega}\cap \Gamma=\emptyset$ or $L_s^{\omega}\subseteq \Gamma$.}

\begin{theorem}
	\label{NewET}
	Assume that the random periodic path $Y$ satisfies Condition A and the periodic measure $\rho.: \mathbb{R}\rightarrow \mathcal{P}(\mathbb{X})$ is given in Theorem \ref{PM2}. Then if the dynamical system $(\Omega, \mathcal{F}, P, (\theta_{\tau}^n)_{n\geq 0})$ is ergodic, the $\tau$-periodic measure $\{\rho_s\}_{s\in \mathbb{R}}$ defined in (\ref{nkey1}) is PS-ergodic, and hence $\bar{\rho}$ defined in \eqref{time average measure rho} is ergodic.
\end{theorem}
\begin{proof}
	Since $(\Omega, \mathcal{F}, P, (\theta_{\tau}^n)_{n\geq 0})$ is ergodic, by Theorem \ref{ET} we know that $(\bar{\Omega}, \bar{\mathcal{F}}, \mu_{s}, (\bar{\Theta}_{\tau}^n)_{n\geq 0})$ is ergodic. By definition of PS-ergodic, we will show that $\rho_{s}$ is ergodic associated with the discrete Markovian semigroup $\{P_{\tau}^k\}_{k\in \mathbb{N}}$. Equivalently, we need to show that for any $\Gamma\in \mathcal{B}(\mathbb{X})$, if $P_{\tau}I_{\Gamma}=I_{\Gamma}, \rho_{s}-a.s.$, then $\rho_{s}(\Gamma)=0$ or 1.
	
	For any $A\in \bar{\mathcal{F}}$ and $\Gamma\in \mathcal{B}(\mathbb{X})$, since $\mu_{s}(A)=P(\{\omega: (\omega, Y(s, \theta_{-s}\omega))\in A\})$ and $\rho_{s}(\Gamma)=P(\{\omega: Y(s, \theta_{-s}\omega)\in \Gamma\})$, then $\rho_{s}(\Gamma)=\mu_{s}(\Omega\times \Gamma)$.
	Next we consider the subset $\Omega\times \Gamma$ in $\bar{\Omega}$.
	Note
	\begin{equation*}
	\begin{split}
	\bar{\Theta}_{\tau}^{-1}(\Omega\times \Gamma)
	      &=\{(\omega, x): \bar{\Theta}_{\tau}(\omega, x)\in \Omega\times \Gamma\}\\
	      &=\{(\omega, x): (\theta_{\tau}\omega, \Phi(\tau, \omega)x)\in \Omega\times \Gamma\}\\
	      &=\{(\omega, x): \Phi(\tau, \omega)x \in \Gamma\}.\\
	\end{split}
	\end{equation*}
	Define $\Omega_{\tau, \Gamma}^x:=\{\omega: \Phi(\tau, \omega)x\in \Gamma\}$, then we have
	\begin{equation}
	\label{eq1}
    \begin{split}
    \bar{\Theta}_{\tau}^{-1}(\Omega\times \Gamma)
          =\bigcup_{x\in \mathbb{X}}(\Omega_{\tau, \Gamma}^x\times \{x\})
          =\left(\bigcup_{x\in \Gamma}(\Omega_{\tau, \Gamma}^x\times \{x\})\right)\bigcup \left(\bigcup_{x\notin \Gamma}(\Omega_{\tau, \Gamma}^x\times \{x\})\right).
    \end{split}
    \end{equation}	
    It is easy to see that $\bigcup_{x\in \Gamma}(\Omega_{\tau, \Gamma}^x\times \{x\})$ is a subset of $\Omega\times \Gamma$, and
    \begin{equation*}
    \begin{split}
    \mu_{s}\left(\bigcup_{x\in \Gamma}(\Omega_{\tau, \Gamma}^x\times \{x\})\right)
         &=\int_{\Omega\times \mathbb{X}}I_{\cup_{x\in \Gamma}(\Omega_{\tau, \Gamma}^x\times \{x\})}(\omega, x)\delta_{Y(s, \theta_{-s}\omega)}(dx)P(d\omega)\\
         &=\int_{\Omega}I_{\cup_{x\in \Gamma}(\Omega_{\tau, \Gamma}^x\times \{x\})}(\omega, Y(s,\theta_{-s}\omega))P(d\omega)\\
         &=P\left(\{\omega: (\omega, Y(s,\theta_{-s}\omega))\in\bigcup_{x\in \Gamma}(\Omega_{\tau, \Gamma}^x\times \{x\})\}\right)\\
         &=P\left(\{\omega: Y(s,\theta_{-s}\omega)\in \Gamma, \omega \in \Omega_{\tau, \Gamma}^{Y(s, \theta_{-s}\omega)}\}\right)\\
         &=P\left(\{\omega: Y(s,\theta_{-s}\omega)\in \Gamma, \Phi(\tau, \omega)Y(s,\theta_{-s}\omega)\in \Gamma\}\right)\\
         &=P\left(\{\omega: Y(s,\theta_{-s}\omega)\in \Gamma, Y(s+\tau,\theta_{-s}\omega)\in \Gamma\}\right).\\
    \end{split}
    \end{equation*}
    Suppose now that $\Gamma\in \mathcal{I}_s^{\tau}$, and Condition A holds, we have
    $$P\left(\{\omega: Y(s,\theta_{-s}\omega)\in \Gamma, Y(s+\tau,\theta_{-s}\omega)\in \Gamma\}\right)=P\left(\{\omega: Y(s,\theta_{-s}\omega)\in \Gamma\}\right).$$
    So
    \begin{equation*}
    \begin{split}
    \mu_{s}\left(\bigcup_{x\in \Gamma}(\Omega_{\tau, \Gamma}^x\times \{x\})\right)
         =P\left(\{\omega: Y(s,\theta_{-s}\omega)\in \Gamma\}\right)
         =\rho_{s}(\Gamma)
         =\mu_{s}(\Omega\times \Gamma).
    \end{split}
    \end{equation*}
    Hence
    \begin{equation}
    \label{*}
    	\mu_{s}\left((\Omega\times \Gamma)\setminus \bigcup_{x\in \Gamma}(\Omega_{\tau, \Gamma}^x\times \{x\})\right)=0.
    \end{equation}
    Similarly, for the set $\bigcup_{x\notin \Gamma}(\Omega_{\tau, \Gamma}^x\times \{x\})$,
    \begin{equation}
    \begin{split}
    \label{**}
    \mu_{s}\left(\bigcup_{x\notin \Gamma}(\Omega_{\tau, \Gamma}^x\times \{x\})\right)
    &=P\left(\{\omega: (\omega, Y(s,\theta_{-s}\omega))\in\bigcup_{x\notin \Gamma}(\Omega_{\tau, \Gamma}^x\times \{x\})\}\right)\\
    &=P\left(\{\omega: Y(s,\theta_{-s}\omega)\notin \Gamma, \omega \in \Omega_{\tau, \Gamma}^{Y(s, \theta_{-s}\omega)}\}\right)\\
    &=P\left(\{\omega: Y(s,\theta_{-s}\omega)\notin \Gamma, Y(s+\tau,\theta_{-s}\omega)\in \Gamma\}\right)\\
    &=0.\\
    \end{split}
    \end{equation}
    The last equality is due to Condition A. Since it follows from (\ref{eq1}) that
    \begin{equation*}
    \begin{split}
     (\Omega\times \Gamma) \bigtriangleup \bar{\Theta}_{\tau}^{-1}(\Omega\times \Gamma)
        &=\left((\Omega\times \Gamma)\setminus \bar{\Theta}_{\tau}^{-1}(\Omega\times \Gamma)\right)\cup \left(\bar{\Theta}_{\tau}^{-1}(\Omega\times \Gamma)\setminus (\Omega\times \Gamma)\right)\\
        &\subseteq \left((\Omega\times \Gamma)\setminus \bigcup_{x\in \Gamma}(\Omega_{\tau, \Gamma}^x\times \{x\})\right)\bigcup \left(\bigcup_{x\notin \Gamma}(\Omega_{\tau, \Gamma}^x\times \{x\})\right),\\
    \end{split}
    \end{equation*}
    then from (\ref{*}) and (\ref{**}), we have
    $$\mu_{s}((\Omega\times \Gamma) \bigtriangleup \bar{\Theta}_{\tau}^{-1}(\Omega\times \Gamma))=0.$$
    Applying equivalent condition of the ergodicity of the dynamical system $(\bar{\Omega}, \bar{\mathcal{F}}, \mu_{s}, (\bar{\Theta}_{\tau}^n)_{n\geq 0})$ (c.f. \cite{Walters}), we have
    $$\rho_{s}(\Gamma)=\mu_{s}(\Omega\times \Gamma)=0 \text{ or } 1.$$
    So  $\rho_{s}$ is ergodic associated with the discrete Markovian semigroup $\{P_{\tau}^k\}_{k\in \mathbb{N}}$ for all $s\in \mathbb{R}$ and therefore $\{\rho_{s}\}_{s\in \mathbb{R}}$ is PS-ergodic. Finally by Theorem 1.3.6 in \cite{F.Z2}, we know that $\bar{\rho}$ is ergodic.
\end{proof}

\begin{theorem}
	\label{Add1}
	Assume Condition P. If the $\tau$-periodic measure $\{\rho_{s}\}_{s\in \mathbb{R}}$ defined in (\ref{nkey1}) is PS-ergodic, then for any given $s\in \mathbb{R}, \Gamma \in \mathcal{I}_s^{\tau}$, we have either for $P$-almost all $\omega$, $L_s^{\omega}\cap \Gamma=\emptyset$ or for $P$-almost all $\omega$, $L_s^{\omega}\subseteq \Gamma$.
\end{theorem}

\begin{proof}
	Since $\{\rho_{s}\}_{s\in \mathbb{R}}$ is PS-ergodic, then for any given $s\in \mathbb{R}$ and $\Gamma\in \mathcal{I}_s^{\tau}$, we have $\rho_{s}(\Gamma)=0$ or 1. By Theorem \ref{PM2}, we have $\rho_s(\Gamma)=P(\{\omega: Y(s, \omega)\in \Gamma\})$.
	
	Let $\Lambda=\{\omega: Y(s, \omega)\in \Gamma\}$ and $\tilde{\Lambda}=\bigcup_{k\in \mathbb{Z}}\theta_{k\tau}^{-1}\Lambda$. It is easy to see that for any $\omega \in \tilde{\Lambda}^c$, $\theta_{k\tau}\omega \in \tilde{\Lambda}^c$ for all $k\in \mathbb{Z}$. Thus for any $\omega \in \tilde{\Lambda}^c$
	\begin{equation*}
		\begin{split}
		L_s^{\omega}&=\{Y(s+k\tau, \omega): k\in \mathbb{Z}\}\\
		                    &=\{Y(s, \theta_{k\tau}\omega): k\in \mathbb{Z}\}\\
		                    &\subseteq \bigcup_{\omega\in \tilde{\Lambda}^c}\{Y(s, \omega)\}.\\
		\end{split}
	\end{equation*}
	By the definition of $\Lambda$ and $\tilde{\Lambda}$, we have 
	$$\bigcup_{\omega\in \tilde{\Lambda}^c}\{Y(s, \omega)\} \subseteq \bigcup_{\omega\in \Lambda^c}\{Y(s, \omega)\} \subseteq \Gamma^c.$$ 
	This means $L_s^{\omega}\cap \Gamma=\emptyset$ for all $\omega \in \tilde{\Lambda}^c$, i.e. $\tilde{\Lambda}^c \subset \{\omega: L_s^{\omega}\cap \Gamma = \emptyset\}$. However, if $\rho_s(\Gamma)=0$, then $P(\Lambda)=\rho_s(\Gamma)=0$. It follows that $P(\tilde{\Lambda})=0$ and $P(\tilde{\Lambda}^c)=1$. Finally $P(\{\omega: L_s^{\omega}\cap \Gamma=\emptyset\})=1$.
	
	Similarly, let $\Sigma=\{\omega: Y(s, \omega)\notin \Gamma\}$ and $\tilde{\Sigma}=\bigcup_{k\in \mathbb{Z}}\theta_{k\tau}^{-1}\Sigma$. It is also easy to see that for all $\omega \in \tilde{\Sigma}^c$, then $\theta_{k\tau}\omega \in \tilde{\Sigma}^c$ for all $k\in \mathbb{Z}$ and
	$$L_s^{\omega}\subseteq \bigcup_{\omega \in \tilde{\Sigma}^c}\{Y(s, \omega)\} \subseteq \bigcup_{\omega \in \Sigma^c}\{Y(s,\omega)\} \subseteq \Gamma.$$
	This means $L_s^{\omega}\subseteq \Gamma$ for all $\omega \in \tilde{\Sigma}^c$, i.e. $\tilde{\Sigma}^c\subset \{\omega: L_s^{\omega}\subseteq \Gamma\}$. However, when $\rho_{s}(\Gamma)=1$, $P(\Sigma)=1-P(\Sigma^c)=1-\rho_s(\Gamma)=0$. Thus $P(\tilde{\Sigma})=0$ and $P(\tilde{\Sigma}^c)=1$. It then follows that $P(\{\omega: L_s^{\omega}\subseteq \Gamma\})=1$.
\end{proof}

\begin{remark}
	\label{remark}
	The result of Theorem \ref{Add1} is stronger than Condition A. This can be seen as follows: The statement that either for P-almost all $\omega$, $L_s^{\omega}\cap \Gamma=\emptyset$ or for P-almost all $\omega$, $L_s^{\omega}\subseteq \Gamma$ means either there exists $\Omega_{1}\subset \Omega$ with $P(\Omega_{1})=1$ such that when $\omega\in\Omega_{1}$, $L_s^{\omega}\cap \Gamma=\emptyset$ or there exists $\Omega_{2}\subset \Omega$ with $P(\Omega_{2})=1$ such that when $\omega\in\Omega_{2}$, $L_s^{\omega}\subseteq \Gamma$. In the first case, let $\Omega_0=\{\omega: L_s^{\omega}\cap \Gamma=\emptyset\}$. Then $\Omega_0\supset\Omega_{1}$, so $P(\Omega_0)=1$. It is obvious that when $\omega\in\Omega_0, L_s^{\omega}\cap \Gamma=\emptyset$. In the second case, let $\Omega_0=\{\omega: L_s^{\omega}\subseteq \Gamma\}$. Then $\Omega_0\supset\Omega_{2}$, so $P(\Omega_0)=1$. It is obvious that when $\omega\in\Omega_0, L_s^{\omega}\subseteq \Gamma$. In both cases, there exists $\Omega_0$ with $P(\Omega_0)=1$ such that when $\omega\in\Omega_0$, either $L_s^{\omega}\cap \Gamma=\emptyset$ or $L_s^{\omega}\subseteq \Gamma$.
\end{remark}

\begin{corollary}
	If the dynamical system $(\Omega, \mathcal{F}, P, (\theta_{\tau}^n)_{n\geq 0})$ is ergodic, then Condition A and the $\tau$-periodic measure $\{\rho_{s}\}_{s\in \mathbb{R}}$ being PS-ergodic are equivalent. In this case, the statement that P-a.s. either $L_s^{\omega}\cap \Gamma = \emptyset$ or $L_s^{\omega}\subseteq \Gamma$ and the statement that either P-a.s. $L_s^{\omega}\cap \Gamma = \emptyset$ or P-a.s. $L_s^{\omega}\subseteq \Gamma$ are equivalent.
\end{corollary}
\begin{proof}
	This corollary can be easily obtained from Theorem \ref{NewET}, Theorem \ref{Add1} and Remark \ref{remark}.
\end{proof}

\subsection{Ergodicity of canonical dynamical system generated from invariant measure: sufficient condition}
Next we consider a stationary path $Y: \Omega\rightarrow \mathbb{X}$ of the Markovian cocycle $\Phi$ and the invariant measure $\mu$ defined in (\ref{IMmu}). Then we recall that the measure $\rho\in\mathcal{P}(\mathbb{X})$ defined by 
 \begin{equation}
 \label{eq3}
 	\rho:=E_P(\mu).=E_P\delta_{Y(\cdot)},
 \end{equation}
which is the law of the stationary path $Y$, is an invariant measure and satisfies that $\rho(\Gamma)=P(\{\omega: Y(\omega)\in \Gamma\})$, for any $\Gamma \in \mathcal{B}(\mathbb{X})$.

Now we define
$$\tilde{L}^{\omega}:=\{Y(\theta_{s}\omega): s\in \mathbb{R}\},$$
and the $\rho$ invariant set associated with the Markovian semigroup $(P_t)_{t\geq 0}$
$$\mathcal{I}:=\{\Gamma\in \mathcal{B}(\mathbb{X}): P_tI_{\Gamma}=I_{\Gamma}, \rho-a.s. \text{ for all } t\geq 0\}.$$
\\
For these two sets, we consider\\
\textbf{Condition $\text{A}^{\prime}$.}  \textit{For any $\Gamma\in \mathcal{I}$, one has for $P$-almost all $\omega\in \Omega$, $\tilde{L}^{\omega}\cap \Gamma=\emptyset$ or $\tilde{L}^{\omega}\subseteq \Gamma$.}

\begin{theorem}
	\label{SME}
	Assume that the stationary path $Y$ satisfies Condition $A^{\prime}$ and let $\rho \in \mathcal{P}(\mathbb{X})$ be the invariant measure given in (\ref{eq3}). If the dynamical system $(\Omega, \mathcal{F}, P, (\theta_t)_{t\geq 0})$ is ergodic, then $\rho$ is ergodic.
\end{theorem}

To prove Theorem \ref{SME} we need the following lemma, which is of interest in its own right. We present it using the same notation as a metric dynamical system. But it does not have to link with the metric dynamical systems of a random dynamical system. It is true for 
a setting of continuous dynamical system of a probability space (measure space).

\begin{lemma}
	\label{lemma}
	Assume that $(\Omega, \mathcal{F}, P, (\theta_{t})_{t\geq 0})$ is a dynamical system, then the following two statements are equivalent:
	\begin{enumerate}[(i)]
	   \item $(\Omega, \mathcal{F}, P, (\theta_{t})_{t\geq 0})$ is ergodic.
	   \item If $A \in \mathcal{F}$, $\bigcup_{t\geq T}\theta_t^{-1}A\in \mathcal{F}$ and $P\left((\bigcup_{t\geq T}\theta_t^{-1}A)\bigtriangleup A\right)=0$ for any $T\geq 0$, then $P(A)=0$ or 1.
\end{enumerate}
\end{lemma}

\begin{proof}
	$(i)\Rightarrow (ii)$. Assume $A\in \mathcal{F}$, $\bigcup_{t\geq T}\theta_t^{-1}A\in \mathcal{F}$ and $P\left((\bigcup_{t\geq T}\theta_t^{-1}A)\bigtriangleup A\right)=0$ for any $T\geq 0$. Define
	$$A_{\infty}:=\bigcap_{T\geq 0}\bigcup_{t\geq T}\theta_t^{-1}A,$$
	we know that 
	$$A_{\infty}=\lim_{T\rightarrow \infty}\bigcup_{t\geq T}\theta_t^{-1}A=\lim_{n\rightarrow \infty}\bigcup_{t\geq n}\theta_t^{-1}A\in \mathcal{F}.$$
	Then it is easy to see that for all $s\geq 0$,
	$$\theta_s^{-1}A_{\infty}=\bigcap_{T\geq 0}\bigcup_{t\geq T+s}\theta_t^{-1}A=A_{\infty}.$$
	Thus $A_{\infty}$ is an invariant set. By the ergodicity assumption, we have
	\begin{equation}
		\label{AIM}
		P(A_{\infty})=0 \text{ or } P(A_{\infty})=1.
	\end{equation}
	For any $T \geq 0$, since we have
	$$P\left((\bigcup_{t\geq T}\theta_t^{-1}A)\bigtriangleup A\right)=0,$$
	then
	$$P\left((\bigcup_{t\geq T}\theta_t^{-1}A)\setminus A\right)=0
	\text{ and }
	P\left(A\setminus \bigcup_{t\geq T}\theta_t^{-1}A\right)=0.$$
	But note that as $T \rightarrow \infty$, we have
	$$(\bigcup_{t\geq T}\theta_t^{-1}A)\setminus A\downarrow A_{\infty}\setminus A
	\text{ and }
	A\setminus \bigcup_{t\geq T}\theta_t^{-1}A\uparrow A\setminus A_{\infty}.$$
	By continuity of measure, we have
	$$P(A_{\infty}\setminus A)=0
	\text{ and }
	P(A\setminus A_{\infty})=0.$$
	Now recall (\ref{AIM}). Consider the case $P(A_{\infty})=0$, then
	$$P(A)=P(A)-P(A_{\infty})\leq P(A\setminus A_{\infty})=0.$$
	Now consider the case that $P(A_{\infty})=1$, then
	$$1-P(A)=P(A_{\infty})-P(A)\leq P(A_{\infty}\setminus A)=0.$$
	Thus the assertion $(ii)$ is proved.

	$(ii)\Rightarrow (i)$. Assume $A\in \mathcal{F}$ and $\theta_t^{-1}A=A$ for all $t\geq 0$. Then we have for all $T\geq 0$
	$$\bigcup_{t\geq T}\theta_t^{-1}A=A\in \mathcal{F},$$
	and
	$$P\left((\bigcup_{t\geq T}\theta_t^{-1}A)\bigtriangleup A\right)=P(\emptyset)=0.$$
	By assertion $(ii)$, we have
	$$P(A)=0 \text{ or } 1.$$
	Thus the assertion $(i)$ is proved.
\end{proof}

Next we  will give the proof of Theorem \ref{SME}.\\
\\
	\textbf{Proof of Theorem \ref{SME}}.
	Since $(\Omega, \mathcal{F}, P, (\theta_{t})_{t\geq 0})$ is ergodic, by Theorem \ref{SPET} we know that $(\bar{\Omega}, \bar{\mathcal{F}}, \mu, (\bar{\Theta}_t)_{t\geq 0})$ is ergodic. Next we just need to show that for any $\Gamma\in \mathcal{B}(\mathbb{X})$, if $P_tI_{\Gamma}=I_{\Gamma}, \text{ for all } t\geq 0, \rho-a.s.$, then $\rho(\Gamma)=0$ or 1.
	
	Similar to the proof of Theorem \ref{NewET},  for any $A\in \bar{\mathcal{F}}$ and $\Gamma\in \mathcal{B}(\mathbb{X})$, since $\mu(A)=P(\{\omega: (\omega, Y(\omega))\in A\})$ and $\rho(\Gamma)=P(\{\omega: Y( \omega)\in \Gamma\})$, then $\rho(\Gamma)=\mu(\Omega\times \Gamma)$.
	Next we consider the subset $\Omega\times \Gamma$ in $\bar{\Omega}$.
	Note that
	\begin{equation*}
	\begin{split}
	\bar{\Theta}_{t}^{-1}(\Omega\times \Gamma)
	&=\{(\omega, x): \bar{\Theta}_{t}(\omega, x)\in \Omega\times \Gamma\}\\
	&=\{(\omega, x): (\theta_{t}\omega, \Phi(t, \omega)x)\in \Omega\times \Gamma\}\\
	&=\{(\omega, x): \omega \in \Omega, \Phi(t, \omega)x \in \Gamma\}.\\
	\end{split}
	\end{equation*}
	Define $\Omega_{t, \Gamma}^x:=\{\omega: \Phi(t, \omega)x\in \Gamma\}$. Then we have
	\begin{equation*}
	\begin{split}
	\bar{\Theta}_{t}^{-1}(\Omega\times \Gamma)
	&=\bigcup_{x\in \mathbb{X}}(\Omega_{t, \Gamma}^x\times \{x\})\\
	&=\left(\bigcup_{x\in \Gamma}(\Omega_{t, \Gamma}^x\times \{x\})\right)\bigcup \left(\bigcup_{x\notin \Gamma}(\Omega_{t, \Gamma}^x\times \{x\})\right).\\
	\end{split}
	\end{equation*}
	Now for any $T \geq 0$,
	$$\bigcup_{t\geq T}\bar{\Theta}_{t}^{-1}(\Omega\times \Gamma)=\left(\bigcup_{t\geq T}\bigcup_{x\in \Gamma}(\Omega_{t, \Gamma}^x\times \{x\})\right)\bigcup \left(\bigcup_{t\geq T}\bigcup_{x\notin \Gamma}(\Omega_{t, \Gamma}^x\times \{x\})\right).$$
	Again, it is easy to see that $\bigcup_{t\geq T}\bigcup_{x\in \Gamma}(\Omega_{t, \Gamma}^x\times \{x\})$ is a subset of $\Omega\times \Gamma$, and
	\begin{equation*}
	\begin{split}
	\mu\left(\bigcup_{t\geq T}\bigcup_{x\in \Gamma}(\Omega_{t, \Gamma}^x\times \{x\})\right)
	&=\int_{\Omega\times \mathbb{X}}I_{\cup_{t\geq T}\cup_{x\in \Gamma}(\Omega_{t, \Gamma}^x\times \{x\})}(\omega, x)\delta_{Y(\omega)}(dx)P(d\omega)\\
	&=\int_{\Omega}I_{\cup_{t\geq T}\cup_{x\in \Gamma}(\Omega_{t, \Gamma}^x\times \{x\})}(\omega, Y(\omega))P(d\omega)\\
	&=P\left(\{\omega: (\omega, Y(\omega))\in\bigcup_{t\geq T}\bigcup_{x\in \Gamma}(\Omega_{t, \Gamma}^x\times \{x\})\}\right)\\
	&=P\left(\{\omega: Y(\omega)\in \Gamma \text{ and } \omega \in \Omega_{t_0, \Gamma}^{Y(\omega)} \text{ for some } t_0\geq T\}\right)\\
	&=P\left(\{\omega: Y(\omega)\in \Gamma \text{ and }\Phi(t_0, \omega)Y(\omega)\in \Gamma \text{ for some } t_0\geq T\}\right)\\
	&=P\left(\{\omega: Y(\omega)\in \Gamma \text{ and } Y(\theta_{t_0}\omega)\in \Gamma \text{ for some } t_0\geq T\}\right).\\
	\end{split}
	\end{equation*}
	Suppose now that $\Gamma\in \mathcal{I}$, and Condition $A^{\prime}$ holds, we have
	$$P\left(\{\omega: Y(\omega)\in \Gamma \text{ and } Y(\theta_{t_0}\omega)\in \Gamma \text{ for some } t_0\geq T\}\right)=P\left(\{\omega: Y(\omega)\in \Gamma\}\right).$$
	Then
	\begin{equation*}
	\begin{split}
	\mu\left(\bigcup_{t\geq T}\bigcup_{x\in \Gamma}(\Omega_{t, \Gamma}^x\times \{x\})\right)
	=P\left(\{\omega: Y(\omega)\in \Gamma\}\right)
	=\rho(\Gamma)
	=\mu(\Omega\times \Gamma).
	\end{split}
	\end{equation*}
	Hence
	\begin{equation}
	\label{***}
		\mu\left((\Omega\times \Gamma)\setminus \bigcup_{t\geq T}\bigcup_{x\in \Gamma}(\Omega_{t, \Gamma}^x\times \{x\})\right)=0. 
    \end{equation}
	Similarly, for the set $\bigcup_{t\geq T}\bigcup_{x\notin \Gamma}(\Omega_{t, \Gamma}^x\times \{x\})$,
	\begin{equation*}
	\begin{split}
	\mu\left(\bigcup_{t\geq T}\bigcup_{x\notin \Gamma}(\Omega_{t, \Gamma}^x\times \{x\})\right)
	&=P\left(\{\omega: (\omega, Y(\omega))\in\bigcup_{t\geq T}\bigcup_{x\notin \Gamma}(\Omega_{t, \Gamma}^x\times \{x\})\}\right)\\
	&=P\left(\{\omega: Y(\omega)\notin \Gamma \text{ and }\omega \in \Omega_{t_0, \Gamma}^{Y(\omega)} \text{ for some } t_0\geq T\}\right)\\
	&=P\left(\{\omega: Y(\omega)\notin \Gamma \text{ and }Y(\theta_{t_0}\omega)\in \Gamma \text{ for some } t_0\geq T\}\right),\\
	\end{split}
	\end{equation*}
	and from Condition $A^{\prime}$, we have
	\begin{equation}
	\label{****}
		\mu\left(\bigcup_{t\geq T}\bigcup_{x\notin \Gamma}(\Omega_{t, \Gamma}^x\times \{x\})\right)=P\left(\{\omega: Y(\omega)\notin \Gamma \text{ and } Y(\theta_t\omega)\in \Gamma \text{ for some } t_0\geq T\}\right)=0.
	\end{equation}
	Now note that
	\begin{equation*}
	\begin{split}
	 &\left(\bigcup_{t\geq T}\bar{\Theta}_t^{-1}(\Omega\times \Gamma)\right) \bigtriangleup (\Omega\times \Gamma)\\
	=&\left(\left(\bigcup_{t\geq T}\bar{\Theta}_t^{-1}(\Omega\times \Gamma)\right)\setminus (\Omega\times \Gamma)\right) \bigcup  \left((\Omega\times \Gamma)\setminus \bigcup_{t\geq T}\bar{\Theta}_t^{-1}(\Omega\times \Gamma)\right)\\
	\subseteq& \left(\bigcup_{t\geq T}\bigcup_{x\notin \Gamma}(\Omega_{t, \Gamma}^x\times \{x\})\right) \bigcup \left((\Omega\times \Gamma)\setminus \bigcup_{t\geq T}\bigcup_{x\in \Gamma}(\Omega_{t, \Gamma}^x\times \{x\})\right).\\
	\end{split}
	\end{equation*}
	Then for all $T\geq 0$, from (\ref{***}) and (\ref{****}), we have
	$$\mu\left( \left(\bigcup_{t\geq T}\bar{\Theta}_t^{-1}(\Omega\times \Gamma)\right) \bigtriangleup (\Omega\times \Gamma)\right)=0.$$
	Applying the ergodicity of the dynamical system $(\bar{\Omega}, \bar{\mathcal{F}}, \mu, (\bar{\Theta}_t)_{t\geq 0})$ and Lemma \ref{lemma}, we have
	$$\rho(\Gamma)=\mu(\Omega\times \Gamma)=0 \text{ or } 1.$$
	Therefore, $\rho$ is ergodic. $\qquad\qquad\qquad\qquad\qquad\qquad\qquad\qquad\qquad\qquad\qquad\qquad\qquad\qquad\qquad\qquad \square$
	
	\section{Sublinear dynamical systems from periodic measures}
	In this section, we also assume Condition P. We will give a construction of upper expectations via the periodic measures $\mu.$ and $\rho.$ defined in (\ref{PMRDS}) and (\ref{nkey1}) respectively. Then we can study the ergodicity of the sublinear expectation dynamical system and sublinear canonical dynamical system generated by the upper expectations and Markov semigroup $P_t$ defined in (\ref{Markov semigroup}).
	
	First we recall the definition of sublinear expectation space (c.f. \cite{Peng}). Let $\Omega$ be a given set and let $\mathcal{H}$ be a linear space of real valued functions defined on $\Omega$. We suppose that $\mathcal{H}$ satisfies the following two conditions:
	\begin{enumerate}[(1)]
		\item $c\in \mathcal{H}$ for each constant $c$;
		\item $|X|\in \mathcal{H}$ if $X\in \mathcal{H}$.
	\end{enumerate}
	
	\begin{definition} (\cite{Peng})
		A sublinear expectation $\mathbb{E}$ is a functional $\mathbb{E}: \mathcal{H}\to \mathbb{R}$ satisfying
		\begin{enumerate}[(i)]
			\item Monotonicity: 
			$$\mathbb{E}[X]\leq \mathbb{E}[Y] \text{ if } X\leq Y.$$
			\item Constant preserving:
			$$\mathbb{E}[c]=c \text{ for } c\in \mathbb{R}.$$
			\item Sub-additivity: For each $X, Y\in \mathcal{H}$,
			$$\mathbb{E}[X+Y]\leq \mathbb{E}[X]+\mathbb{E}[Y].$$
			\item Positive homogeneity:
			$$\mathbb{E}[\lambda X]=\lambda \mathbb{E}[X] \text{ for } \lambda\geq 0.$$
		\end{enumerate}
	\end{definition}
	
	The triplet $(\Omega,\mathcal{H},\mathbb{E})$ is called a sublinear expectation space. If $(i)$ and $(ii)$ are satisfied, $\mathbb{E}$ is called a nonlinear expectation and the triplet $(\Omega,\mathcal{H},\mathbb{E})$ is called a nonlinear expectation space.
	\subsection{Ergodic sublinear dynamical system on upper expectation space}
	
	Recall the product space $(\bar{\Omega}, \bar{\mathcal{F}})=(\Omega \times \mathbb{X}, \mathcal{F}\otimes \mathcal{B}(\mathbb{X}))$ in Section 2 and define
	\begin{equation}
	\mathbb{E}[\varphi]=\sup_{s\in [0,\tau)}E_{\mu_{s}}[\varphi], \text{ for all } \varphi \in L_{b}(\bar{\Omega}),
	\end{equation}
	where $L_{b}(\bar{\Omega}):=\{\varphi: \bar{\Omega}\rightarrow \mathbb{R}| \varphi \text{ is measurable and bounded} \}$, $\mu_s$ is the periodic measure defined in (\ref{PMRDS}), and
	\begin{equation*}
	\begin{split}
	E_{\mu_{s}}[\varphi]
	&:=\int_{\Omega\times \mathbb{X}}\varphi(\omega, x)(\mu_{s})_{\omega}(dx)P(d\omega)\\
	&=\int_{\Omega\times \mathbb{X}}\varphi(\omega, x)\delta_{Y(s, \theta_{-s}\omega)}(dx)P(d\omega)\\
	&=\int_{\Omega}\varphi(\omega, Y(s, \theta_{-s}\omega))P(d\omega).\\
	\end{split}
	\end{equation*}
	It is easy to verify that $\mathbb{E}$ is a sublinear expectation on $(\bar{\Omega}, L_{b}(\bar{\Omega}))$.
	Recall the definition (c.f. \cite{F.Z1})
	\begin{equation}
	\label{key3}
	\bar{\Theta}_t\mathbb{E}[\varphi(\cdot)]:=\mathbb{E}[\varphi(\bar{\Theta}_t\cdot)].
	\end{equation}
	
	\begin{proposition}
		\label{ISE}
		Assume Condition P. The skew product $(\bar{\Theta}_{t})_{t\geq 0}$ preserves the sublinear expectation $\mathbb{E}$.
	\end{proposition}	
	\begin{proof}
		We just need to prove $\bar{\Theta}_{t}\mathbb{E}[\varphi]=\mathbb{E}[\varphi], \text{ for all } \varphi \in L_{b}(\bar{\Omega})$. To see the proof, for any $s\in \mathbb{R}$, by definition of $E_{\mu_{s}}$, we have
		\begin{equation*}
		\begin{split}
		\bar{\Theta}_{t}E_{\mu_{s}}[\varphi]
		&=E_{\mu_{s}}[\varphi\circ \bar{\Theta}_{t}]\\
		&=\int_{\Omega\times \mathbb{X}}\varphi(\bar{\Theta}_{t}(\omega, x))\delta_{Y(s, \theta_{-s}\omega)}(dx)P(d\omega)\\
		&=\int_{\Omega\times \mathbb{X}}\varphi(\theta_{t}\omega,\Phi(t, \omega) x)\delta_{Y(s, \theta_{-s}\omega)}(dx)P(d\omega)\\
		&=\int_{\Omega}\varphi(\theta_{t}\omega, \Phi(t, \omega)Y(s, \theta_{-s}\omega))P(d\omega)\\
		&=\int_{\Omega}\varphi(\theta_{t}\omega, Y(t+s, \theta_{-s}\omega))P(d\omega)\\
		&=\int_{\Omega}\varphi(\theta_{t}\omega, Y(t+s, \theta_{-(t+s)}\theta_{t}\omega))P(d\omega)\\
		&=\int_{\Omega}\varphi(\omega, Y(t+s, \theta_{-(t+s)}\omega))P(d\omega)\\
		&=E_{\mu_{t+s}}[\varphi].\\
		\end{split}
		\end{equation*}
		Then by definition of $\mathbb{E}$, the periodic property of $\mu.$ and (\ref{key3}), we have
		\begin{equation*}
		\begin{split}
		\bar{\Theta}_{t}\mathbb{E}[\varphi]
		=\sup_{s\in [0,\tau)}\bar{\Theta}_{t}E_{\mu_{s}}[\varphi]
		=\sup_{s\in [0,\tau)}E_{\mu_{t+s}}[\varphi]
		=\sup_{s\in [0,\tau)}E_{\mu_{s}}[\varphi]
		=\mathbb{E}[\varphi].
		\end{split}
		\end{equation*}
	\end{proof}
	
	Recall the definition of ergodicity of a sublinear expectation dynamical system.
	\begin{definition}
		\label{D1}
		(\cite{F.Z1}) Let $(\Omega,\mathcal{H},\mathbb{E})$ be a sublinear expectation space and the measurable transformation $\theta: [0,\infty)\times \Omega\to \Omega$ preserve the expectation $\mathbb{E}$. We say that the sublinear expectation dynamical system $(\Omega,\mathcal{H},\mathbb{E},(\theta_t)_{t\geq 0})$ is ergodic if for any $B\in \sigma(\mathcal{H})$ with $\theta_t^{-1}B=B, \text{ for all } t\geq 0$, then $\mathbb{E}[I_B]=0$ or $\mathbb{E}[I_{B^c}]=0$.
	\end{definition}
	Consider the dynamical system $(\bar{\Omega}, L_{b}(\bar{\Omega}), \mathbb{E}, (\bar{\Theta}_{t})_{t\geq 0})$ defined as above, we have
	\begin{theorem}
		\label{ENES}
		Assume Condition P. If $(\bar{\Omega}, \bar{\mathcal{F}}, \mu_{s}, (\bar{\Theta}_{\tau}^n)_{n\geq 0})$ is ergodic for some $s\in\mathbb{R}$, then $(\bar{\Omega}, L_{b}(\bar{\Omega}), \mathbb{E}, (\bar{\Theta}_{t})_{t\geq 0})$ is an ergodic sublinear expectation dynamical system.
	\end{theorem}
	
	\begin{proof}
		Firstly, we know that $\bar{\Theta}_{t}$ preserves the sublinear expectation $\mathbb{E}$ from Proposition \ref{ISE}. For any $A \in \bar{\mathcal{F}}$ with $\bar{\Theta}_{t}^{-1}A=A, \text{ for all } t\geq 0$, by Definition \ref{D1}, we need to prove $V(A)=0$ or $V(A^{c})=0$ where $V(A):=\sup_{s\in [0,\tau)}\mu_{s}(A)$. Since $(\bar{\Omega}, \bar{\mathcal{F}}, \mu_{s}, (\bar{\Theta}_{\tau}^n)_{n\geq 0})$ is ergodic for some $s\in\mathbb{R}$, by Theorem \ref{Converse ET} we know that $(\Omega, \mathcal{F}, P, (\theta_{\tau}^n)_{n\geq 0})$ is ergodic. Considering the proof of Theorem \ref{ET} we have $\mu_s(A)=\mu_0(A)$ for all $s\in\mathbb{R}$ and $\mu_0(A)=0$ or 1. Then
		$$V(A)=\sup_{s\in [0,\tau)}\mu_{s}(A)=\mu_{0}(A)=0,$$
		or 
		$$V(A^c)=\sup_{s\in [0,\tau)}\mu_{s}(A^c)=\mu_{0}(A^c)=1-\mu_{0}(A)=0.$$
	\end{proof}
	
	We say a statement holds $\mathbb{E}$ quasi surely ($\mathbb{E}-q.s.$) if the statement is true on set $A$ with $\mathbb{E}[I_{A^c}]=0$.
	
	\begin{proposition}
		Assume Condition P and $(\bar{\Omega}, \bar{\mathcal{F}}, \mu_{s}, (\bar{\Theta}_{\tau}^n)_{n\geq 0})$ is ergodic for some $s\in\mathbb{R}$, Then for any $\xi \in L_{b}(\bar{\Omega})$, we have
		$$\lim_{T\rightarrow \infty}\frac{1}{T}\int_{0}^{T}U_{t}\xi dt=\frac{1}{\tau}\int_{0}^{\tau}E_{\mu_{s}}[\xi]ds, \quad \mathbb{E}-q.s.,$$
		where $(U_{t}\xi)(\omega, x)=\xi(\bar{\Theta}_{t}(\omega, x))$.
	\end{proposition}
	
	\begin{proof}
		By (i) of Remark \ref{re1}, we know that $(\bar{\Omega}, \bar{\mathcal{F}}, \mu_{s}, (\bar{\Theta}_{\tau}^{n})_{n\geq 0})$ is ergodic for each $s\in \mathbb{R}$. For any $\xi\in L_b(\bar{\Omega})$, without any loss of generality, we can assume that $\xi\geq 0, \mathbb{E}-q.s.$. Let $\xi_{\tau}:=\frac{1}{\tau}\int_{0}^{\tau}U_{t}\xi dt$, then
		$$\frac{1}{n}\sum_{k=0}^{n-1}U_{\tau}^{k}\xi_{\tau}=\frac{1}{n\tau}\int_{0}^{n\tau}U_{t}\xi dt.$$
		Applying Birkhoff's ergodic theorem (c.f. \cite{Da Prato}) for $\xi_{\tau}$ on $(\bar{\Omega}, \bar{\mathcal{F}}, \mu_{s}, (\bar{\Theta}_{\tau}^{n})_{n\geq 0})$, we have
		$$\lim_{n\rightarrow \infty}\frac{1}{n}\sum_{k=0}^{n-1}U_{\tau}^{k}\xi_{\tau}=E_{\mu_{s}}[\xi_{\tau}], \quad \mu_{s}-a.s..$$
		For any arbitrary $T>0$, let $n_{T}=[\frac{T}{\tau}]$ be the maximal nonnegative integer less than or equal to $\frac{T}{\tau}$. Then $n_{T}\tau\leq T<(n_{T}+1)\tau$ and
		$$\frac{n_{T}}{n_{T}+1}\cdot\frac{1}{n_{T}\tau}\int_{0}^{n_{T}\tau}U_{t}\xi dt\leq \frac{1}{T}\int_{0}^{T}U_{t}\xi dt\leq \frac{n_{T}+1}{n_{T}}\cdot\frac{1}{(n_{T}+1)\tau}\int_{0}^{(n_{T}+1)\tau}U_{t}\xi dt.$$
		Thus 
		$$\lim_{T\rightarrow \infty}\frac{1}{T}\int_{0}^{T}U_{t}\xi dt=E_{\mu_{s}}[\xi_{\tau}] \quad \mu_{s}-a.s..$$
		However,
		\begin{equation*}
		\begin{split}
		E_{\mu_{s}}[\xi_{\tau}]&=E_{\mu_{s}}[\frac{1}{\tau}\int_{0}^{\tau}U_{t}\xi dt]\\
		&=\frac{1}{\tau}\int_{0}^{\tau}E_{\mu_{s}}[U_{t}\xi] dt\\
		&=\frac{1}{\tau}\int_{0}^{\tau}\bar{\Theta}_{t}E_{\mu_{s}}[\xi] dt\\
		&=\frac{1}{\tau}\int_{0}^{\tau}E_{\mu_{s+t}}[\xi] dt\\
		&=\frac{1}{\tau}\int_{0}^{\tau}E_{\mu_{t}}[\xi] dt,\\
		\end{split}
		\end{equation*}
		thus 
		$$\lim_{T\rightarrow \infty}\frac{1}{T}\int_{0}^{T}U_{t}\xi dt=\frac{1}{\tau}\int_{0}^{\tau}E_{\mu_{t}}[\xi] dt \quad \mu_{s}-a.s..$$
		Let $A_{s}$ be the $\mu_{s}$-null set such that 
		$$\lim_{T\rightarrow \infty}\frac{1}{T}\int_{0}^{T}U_{t}\xi dt=\frac{1}{\tau}\int_{0}^{\tau}E_{\mu_{t}}[\xi] dt,  \text{ on }A_{s}^{c}.$$
		Let $A=\cap_{s \in [0, \tau)}A_{s}$, then $V(A)=\sup_{s\in [0,\tau)}\mu_{s}(A)=0$ and
		$$\lim_{T\rightarrow \infty}\frac{1}{T}\int_{0}^{T}U_{t}\xi dt=\frac{1}{\tau}\int_{0}^{\tau}E_{\mu_{t}}[\xi] dt,  \text{ on } A^c.$$
	\end{proof}
    
    From the Condition P and the assumption that the periodic measure $\{\mu_{s}\}_{s\in\mathbb{R}}$ on the product space is PS-ergodic, we obtained the Birkhoff's law of large numbers with the convergence in the sense of quasi-surely. This result is stronger than the Birkhoff ergodic type theorem in the almost sure sense that we can obtain from the ergodic theory of periodic measure (\cite{F.Z2}). This justifies the study of the construction of invariant sublinear expectations from periodic measures.
	
	Next we also give two examples of ergodic sublinear dynamical system.
	\begin{example}(An ergodic sublinear dynamical system with discrete time)
	\label{Example1}
	We consider $\Omega_{1}=[0,1), \theta_{\alpha}: \Omega_{1}\rightarrow \Omega_{1}, \theta_{\alpha}(x)=(x+\alpha) \text{ mod }1$. It is well known that the dynamical system $(\Omega_{1}, \mathcal{B}(\Omega_{1}), P_{1}, (\theta_{\alpha}^{n})_{n\geq 0})$ is ergodic when $\alpha$ is irrational. Here $P_{1}$ is the Lebesgue measure on $\Omega_{1}$.
	
	Next, we consider $\hat{\Omega}=[0,2)=[0,1)\cup [1,2)=\Omega_{1}\cup \Omega_{2}$, where $\Omega_{2}=[1,2)$. Define $\hat{\theta}_{\alpha}: \hat{\Omega}\rightarrow \hat{\Omega}$ by
	\begin{equation*}
	\hat{\theta}_{\alpha}(x)=
	\begin{cases}
	\theta_{\alpha}(x)+1, & \text{$x\in \Omega_{1}$},\\
	x-1,                         & \text{$x\in  \Omega_{2}$}.
	\end{cases}
	\end{equation*}
	Then we know that $\hat{\theta}_{\alpha}^2$ maps $\Omega_{i}$ into $\Omega_{i}$ for $i=1,2$. Let $P_{2}$ be the Lebesgue measure on $\Omega_{2}$, then $(\Omega_{i}, \mathcal{B}(\Omega_{i}), P_{i}, (\hat\theta_{\alpha}^{2n}|_{\Omega_{i}})_{n\geq 0}), i=1,2$ are ergodic dynamical systems.
	
	Define 
	$$\bar{P}_{i}(A):=P_{i}(A\cap \Omega_{i}), \text{ for any } A \in \mathcal{B}(\hat{\Omega}),$$
	and
	$$\hat{\mathbb{E}}[X]=E_{\bar{P}_{1}}[X]\vee E_{\bar{P}_{2}}[X], \text{ for all }  X \in L^{1}(\hat{\Omega}).$$
	Then for any $\xi \in L^1(\hat{\Omega})$, we have
	\begin{equation*}
	\begin{split}
	\hat{\theta}_{\alpha}E_{\bar{P}_1}[\xi]
	&=E_{\bar{P}_1}[\xi\circ \hat{\theta}_{\alpha}]\\
	&=\int_{\hat{\Omega}}\xi(\hat{\theta}_{\alpha}(x))\bar{P}_1(dx)\\
	&=\int_{\Omega_1}\xi(\hat{\theta}_{\alpha}(x))P_1(dx)\\
	&=\int_{\Omega_2}\xi(y)P_2(dy)\\
	&=\int_{\hat{\Omega}}\xi(y)\bar{P}_2(dy)\\
	&=E_{\bar{P}_2}[\xi].\\
	\end{split}
	\end{equation*}
	Similarly, 
	$$\hat{\theta}_{\alpha}E_{\bar{P}_2}[\xi]=E_{\bar{P}_1}[\xi].$$
	Thus $\hat{\theta}_{\alpha}\hat{\mathbb{E}}[\xi]= E_{\bar{P}_1}[\xi\circ\hat{\theta}_{\alpha}]\vee E_{\bar{P}_2}[\xi\circ\hat{\theta}_{\alpha}] =E_{\bar{P}_2}[\xi]\vee E_{\bar{P}_1}[\xi] = \hat{\mathbb{E}}[\xi]$, which means $\hat{\theta}_{\alpha}$ preserves the sublinear expectation $\hat{\mathbb{E}}$. Then $(\hat{\Omega}, \mathcal{B}(\hat{\Omega}), (\hat{\theta}_{\alpha}^n)_{n\geq 0}, \hat{\mathbb{E}})$ is a sublinear dynamical system.
	
	Let $\bar{P}=\frac{1}{2}(\bar{P}_1+\bar{P}_2)$. Then $\bar{P}$ is an invariant measure with respect to $\hat{\theta}_{\alpha}$ and ergodic as it is PS-ergodic by Theorem 2.20 in \cite{F.Z2}. Moreover, $\hat{\theta}_{\alpha}$ is also ergodic under the sublinear upper expectation setting that we observe in this section. Here we give a straightforward proof of this result in this special case.	
	\end{example}
	\begin{proposition}
		The sublinear dynamical system $\hat{S}=(\hat{\Omega}, \mathcal{B}(\hat{\Omega}), (\hat{\theta}_{\alpha}^n)_{n\geq 0}, \hat{\mathbb{E}})$ is ergodic while $\alpha$ is irrational.
	\end{proposition}
	
	\begin{proof}
		For any $A \in \mathcal{B}(\hat{\Omega})$ with $\hat{\theta}^{-1}_{\alpha}A=A$, we need to show that 
		$$\hat{\mathbb{E}}[I_A]=\bar{P}_1(A)\vee\bar{P}_2(A)=0 \text{ or } \hat{\mathbb{E}}[I_{A^c}]=\bar{P}_1(A^c)\vee\bar{P}_2(A^c)=0.$$
		By $\hat{\theta}^{-1}_{\alpha}A=A$ we have $(\hat{\theta}^2_{\alpha})^{-1}A=A$, which means
		$$\{(\hat{\theta}^2_{\alpha})^{-1}(A\cap\Omega_1)\}\cup\{(\hat{\theta}^2_{\alpha})^{-1}(A\cap\Omega_2)\}=(A\cap \Omega_1)\cup(A\cap\Omega_2).$$
		Since $\hat{\theta}^2_{\alpha}$ maps $\Omega_{i}$ into $\Omega_{i},i=1,2,$ then $$(\hat{\theta}^2_{\alpha})^{-1}(A\cap\Omega_i)=A\cap\Omega_i.$$
		As we already know, $(\Omega_{i}, \mathcal{B}(\Omega_{i}), ((\hat{\theta}^2_{\alpha})^n)_{n\geq 0}, P_i)$ are ergodic while $\alpha$ is irrational, Then 
		$$P_i(A\cap\Omega_{i})=0 \text{ or } 1, i=1,2.$$
		Since $\hat{\theta}_{\alpha}(A\cap\Omega_{1})=A\cap\Omega_{2}$, then we have $$\bar{P}_2(A)=P_2(A\cap\Omega_{2})=P_2(\hat{\theta}_{\alpha}(A\cap\Omega_{1}))=P_1(A\cap\Omega_{1})=\bar{P}_1(A).$$
		Thus 
		$$\hat{\mathbb{E}}[I_A]=\bar{P}_1(A)\vee\bar{P}_2(A)=0
		\text{ or }
		\hat{\mathbb{E}}[I_A^c]=\bar{P}_1(A^c)\vee\bar{P}_2(A^c)=0.$$
		This means the sublinear dynamical system $\hat{S}=(\hat{\Omega}, \mathcal{B}(\hat{\Omega}), (\hat{\theta}_{\alpha}^n)_{n\geq 0}, \hat{\mathbb{E}})$ is ergodic.
	\end{proof}
	
	\begin{example}(Ergodic sublinear dynamical system on torus)
	We consider the same dynamical system $(\tilde{\Omega}, \tilde{\mathcal{F}}, (\tilde{\theta}^{\alpha}_t)_{t\in \mathbb{R}})$ on torus and periodic probability measure $\mu_s, s\in \mathbb{R},$ as in Example \ref{Example2}. Now we define the upper expectation $\tilde{\mathbb{E}}$ on $L_b(\tilde{\Omega})$ by 
	$$\tilde{\mathbb{E}}[\varphi]:= \sup_{s\in [0,1)}E_{\mu_{s}}[\varphi], \text{ for all } \varphi \in L_b(\tilde{\Omega}).$$
	\end{example}
	\begin{proposition}
		The sublinear dynamical system $(\tilde{\Omega}, L_b(\tilde{\Omega}), \tilde{\mathbb{E}}, (\tilde{\theta}^{\alpha}_t)_{t\geq 0})$ is ergodic.
	\end{proposition}
	\begin{proof}
	  The proof is similar to that of Theorem \ref{ENES}. But in Theorem \ref{ENES}, we assumed that there is a random periodic path. But this proof does not depend on this assumption. The key is $\mu_{s}(\tilde{A})=\mu_{0}(\tilde{A})$ for any invariant set $\tilde{A}$. This is proved by a different method, see Proposition \ref{prop1} and (\ref{eq4}).
	\end{proof}

\subsection{Ergodicity of sublinear canonical dynamical systems with respect to a Markovian semigroup}
Next we will give the definitions of sublinear Markovian systems and their ergodicity (\cite{F.Z1}, \cite{Peng1}).
\begin{definition}
	(\cite{Peng1}) We say that $T: \mathbb{R}^+\times L_{b}(\mathbb{X})\rightarrow L_{b}(\mathbb{X})$ is a sublinear Markovian semigroup if:
	\begin{enumerate}[(i)]
		\item For each fixed $(t,x)\in \mathbb{R}^+\times\mathbb{X}$, $T_{t}[\varphi](x)$ is a sublinear expectation defined on $L_{b}(\mathbb{X})$.
		\item $T_0[\varphi](x)=\varphi(x)$, for each $\varphi\in L_{b}(\mathbb{X})$.
		\item $T_{t}[\varphi]$ satisfies the following Chapman semigroup formula:
		$$T_{t}\circ T_{s}[\varphi]=T_{t+s}[\varphi], \text{ for any }t,s\geq 0.$$
	\end{enumerate}
	Here $L_{b}(\mathbb{X})$ is the set of $\cal B(\mathbb X)$-measurable real-valued function defined on $\mathbb X$ such that $\sup_{x\in \mathbb X}|\phi(x)|<\infty$.
\end{definition}

\begin{remark}
	The Markovian semigroup $P_t, t\geq 0$ given in (\ref{Markov semigroup}) is also a sublinear Markovian semigroup.
\end{remark}

\begin{definition}
	(\cite{F.Z1}) A nonlinear expectation $\tilde{T}: L_{b}(\mathbb{X})\rightarrow \mathbb{R}$ is said to be an invariant nonlinear expectation under a sublinear Markovian semigroup $T_{t}, t\geq 0$ if it  satisfies
	$$T_{t}\tilde{T}=\tilde{T} \text{ for all } t\geq 0.$$
	Here $T_{t}\tilde{T}[\varphi]:=\tilde{T}[T_{t}\varphi].$
\end{definition}

Similarly as in the invariant measure case, denote by $\Omega^*=\mathbb{X}^\mathbb{R}$ the space of all $\mathbb{X}$-valued functions on $\mathbb{R}$, by $\mathcal{F}^*$ the smallest $\sigma$-algebra containing all cylindrical sets of $\Omega^*$, the shift $\theta^*: \mathbb{R}\times \Omega^*\rightarrow \Omega^*$ is defined by $(\theta_{t}^*\omega^*)(s)=\omega^*(t+s)$, for all $\omega^*\in \Omega^*.$
Set 
\begin{equation*}
	\begin{split}
	&L_0(\Omega^*)\\
	&:=\{\xi| \text{there exists } n \geq 1, t_1,\cdots,t_n\in \mathbb{R}, \varphi\in L_b(\mathbb{X}^n) \text{ such that } \xi(\omega^*)=\varphi(\omega^*(t_1),\cdots,\omega^*(t_n))\}.
	\end{split}
\end{equation*}
It is clear that $L_0(\Omega^*)$ is a linear subspace of $L_b(\Omega^*)$.  

For any $\varphi\in L_b(\mathbb{X}^n)$ and $t_1<t_2<\cdots<t_n$, we define $\varphi_i\in L_b(\mathbb{X}^{n-i}), i=1,2,\cdots,n$ as follows:
\begin{equation*}
\begin{split}
\varphi_1(x_1,x_2,\cdots,x_{n-1})&:=T_{t_n-t_{n-1}}[\varphi(x_1,x_2,\cdots,x_{n-1},\cdot)](x_{n-1})\\
\varphi_2(x_1,x_2,\cdots,x_{n-2})&:=T_{t_{n-1}-t_{n-2}}[\varphi_1(x_1,x_2,\cdots,x_{n-2},\cdot)](x_{n-2})\\
&\vdots\\
\varphi_{n-1}(x_1)&:=T_{t_{2}-t_{1}}[\varphi_{n-2}(x_1,\cdot)](x_{1}).\\
\end{split}
\end{equation*}
For any $T_{t}, t\geq 0$ invariant sublinear expectation $\tilde{T}$, we define the sublinear expectation $\mathbb{E}^{\tilde{T}}$ on $L_0(\Omega^*)$ by
$$\mathbb{E}^{\tilde{T}}[\xi]=\tilde{T}[\varphi_{n-1}(\cdot)], \text{ where } \xi(\omega^*)=\varphi(\omega^*(t_1),\omega^*(t_2),\cdots,\omega^*(t_n)).$$
Then we can extend $L_0(\Omega^*)$ into $L_0^p(\Omega^*)$ under the norm $(\mathbb{E}^{\tilde{T}}[|\cdot|^p])^{1/p}, p\geq 1$. Define the space
\begin{equation*}
\begin{split}
&Lip_{b,cyl}(\Omega^*)\\
&:=\{\xi| \text{there exists } n \geq 1, t_1,\cdots,t_n\in \mathbb{R}, \varphi\in C_{b,lip}(\mathbb{X}^n) \text{ such that } \xi(\omega^*)=\varphi(\omega^*(t_1),\cdots,\omega^*(t_n))\}.
\end{split}
\end{equation*}
and $L_G^p(\Omega^*)$ the completion of $Lip_{b,cyl}(\Omega^*)$ under the norm $(\mathbb{E}^{\tilde{T}}[|\cdot|^p])^{1/p}, p\geq 1$. From \cite{F.Z1}, $\theta^*_t$ preserves the expectation $\mathbb{E}^{\tilde{T}}$, and $S^{\tilde{T}}=(\Omega^*, L_G^2(\Omega^*), (\theta^*_t)_{t\in \mathbb{R}}, \mathbb{E}^{\tilde{T}})$ defines a dynamical system, called \textit{canonical dynamical system} associated with $T_t, t\geq 0, \tilde{T},$ and $\theta^*_t$.

\begin{definition}
	(\cite{F.Z1}) The invariant expectation $\tilde{T}$ is said to be ergodic with respect to the sublinear Markovian semigroup $T_t, t\geq 0$, if its associated canonical dynamical system $S^{\tilde{T}}=(\Omega^*, L_G^2(\Omega^*), (\theta^*_t)_{t\geq 0}, \mathbb{E}^{\tilde{T}})$ on the sublinear expectation space is ergodic.
\end{definition}
Next we will consider the ergodic property of a sublinear expectation $\tilde{T}$ on $(\mathbb{X}, L_b(\mathbb{X}))$ defined by
\begin{equation}
\label{eq2}
\tilde{T}[\varphi]=\sup_{s\in [0,\tau)}E_{\rho_{s}}[\varphi], \quad \text{ for all } \varphi\in L_b(\mathbb{X}),
\end{equation}
where $E_{\rho_{s}}[\varphi]:=\int_{\mathbb{X}}\varphi(x)\rho_{s}(dx).$

First we have the following property of $\tilde{T}$.
\begin{proposition}
	The sublinear expectation $\tilde{T}$ is an invariant expectation under the Markovian semigroup $P_t, t\geq 0.$
\end{proposition}
\begin{proof}
	We need to prove that $P_t\tilde{T}[\varphi]=\tilde{T}[\varphi], \text{ for all } t\geq 0, \varphi\in L_b(\mathbb{X}).$ Since $\rho.$ is a periodic measure of the Markovian semigroup $P_t$, we have
	\begin{equation}
	\label{PE}
	\begin{split}
	P_tE_{\rho_s}[\varphi]&=E_{\rho_s}[P_t\varphi]\\
	&=\int_{\mathbb{X}}P_t\varphi(x)\rho_s(dx)\\
	&=\int_{\mathbb{X}}\int_{\mathbb{X}}P_t(x,dy)\varphi(y)\rho_s(dx)\\
	&=\int_{\mathbb{X}}\int_{\mathbb{X}}P_t(x,dy)\rho_s(dx)\varphi(y)\\
	&=\int_{\mathbb{X}}\varphi(y)\rho_{s+t}(dy)\\
	&=E_{\rho_{s+t}}[\varphi].\\
	\end{split}
	\end{equation}
	Therefore by (\ref{eq2}), (\ref{PE}) and the periodicity of $\rho.$, we have that
	\begin{equation*}
	\begin{split}
	P_t\tilde{T}[\varphi]=\tilde{T}[P_t\varphi]
	=\sup_{s\in [0,\tau)}E_{\rho_s}[P_t\varphi]
	=\sup_{s\in [0,\tau)}E_{\rho_{s+t}}[\varphi]
	=\sup_{s\in [0,\tau)}E_{\rho_s}[\varphi]
	=\tilde{T}[\varphi].
	\end{split}
	\end{equation*}
\end{proof}

\begin{theorem}
	Under the assumption in Theorem \ref{NewET}, if $(\Omega, \mathcal{F}, P, (\theta_{\tau}^n)_{n\geq 0})$ is ergodic, then $\tilde{T}$ is ergodic with respect to the Markovian semigroup $P_t, t\geq 0$.
\end{theorem}
\begin{proof}
	We will show that for any $\varphi\in L_b(\mathbb{X})$, if $P_t\varphi=\varphi$ (since $P_t$ is a linear operator, then $P_t(-\varphi)=-\varphi$) for all $t\geq 0$, then $\varphi$ is constant, $\tilde{T}$-q.s. By Theorem 3.25 in \cite{F.Z1}, noting in this part of the theorem in \cite{F.Z1}  that $\varphi(\hat X(0))$ has no mean uncertainty is not needed, we conclude that $\tilde{T}$ is ergodic. 

    Firstly by Theorem \ref{NewET}, we know that $\{\rho_{s}\}_{s\in\mathbb{R}}$ is PS-ergodic. Then by classical ergodic theory (c.f. \cite{Da Prato1}) we have that if $P_{\tau}\varphi=\varphi$, then $\varphi$ is constant, $\rho_{s}$-a.s. for all $s\in\mathbb{R}$. Let $l_s$ be that constant, i.e. $\varphi=l_s, \rho_{s}$-a.s..
	Since $P_t\varphi=\varphi$ for all $t\geq 0$ and by  (\ref{PE}), we have
	\begin{equation*}
	\begin{split}
	l_{s+t}=E_{\rho_{s+t}}[\varphi]
	=E_{\rho_{s}}[P_t\varphi]
	=E_{\rho_{s}}[\varphi]
	=l_s.
	\end{split}
	\end{equation*}
	So $l_s\equiv l_0$ and $\varphi=l_0, \rho_{s}-a.s.$ for all $s\in \mathbb{R}$. Let $A:=\{x: \varphi(x)\neq l_0\}$, then $\rho_{s}(A)=0, \text{ for all } s\in \mathbb{R}$. Thus $\tilde{T}[I_A]=\sup_{s\in [0,\tau)}\rho_{s}(A)=0$, which means $\varphi=l_0, \tilde{T}-q.s.$.
\end{proof}

    \section*{Acknowledgements}
    We are grateful to the anonymous referee for their constructive comments which lead to significant improvements of this paper. We would like to thank Hans Crauel for useful discussions and acknowledge the financial supports of a Royal Society Newton Fund grant NA150344 and an EPSRC grant EP/S005293/1. 

\end{document}